\documentclass[11pt,a4paper,reqno]{amsart}
\usepackage[english]{babel}
\usepackage[T1]{fontenc}
\usepackage{verbatim}
\usepackage{palatino}
\usepackage{amsmath}
\usepackage{mathabx}
\usepackage{amssymb}
\usepackage{amsthm}
\usepackage{amsfonts}
\usepackage{graphicx}
\usepackage[colorlinks = true, citecolor = black]{hyperref}
\usepackage{doi}
\usepackage{esint}
\usepackage{color}
\usepackage{mathtools}
\usepackage{overpic}
\usepackage{enumitem}

\author{}
\title{Structure of sets with nearly maximal Favard length}

\address{Department of Mathematics, Princeton University, Princeton, NJ 08544, USA}
\email{alanchang@math.princeton.edu}

\address{Department of Mathematics and Statistics\\ University of Jyv\"askyl\"a,
P.O. Box 35 (MaD)\\
FI-40014 University of Jyv\"askyl\"a\\
Finland}
\email{damian.m.dabrowski@jyu.fi}
\email{tuomas.t.orponen@jyu.fi}

\address{Research Unit of Mathematical Sciences, University of Oulu, P.O. Box 8000, FI-90014, University of Oulu, Finland}
\email{michele.villa@oulu.fi}
\date{\today}
\author[A. Chang]{Alan Chang}
\author[D. D\k{a}browski]{Damian D\k{a}browski}
\author[T. Orponen]{Tuomas Orponen}
\author[M. Villa]{Michele Villa}
\thanks{D.D. and T.O. are supported by the Academy of Finland via the project \emph{Incidences on Fractals}, grant No. 321896. T.O. is also supported by the Academy of Finland via the project \emph{Quantitative rectifiability in Euclidean and non-Euclidean spaces}, grant Nos. 309365, 314172. M.V. was supported by a starting grant of the University of Oulu.}

\newcommand{\R}{\mathbb{R}}

\newcommand{\N}{\mathbb{N}}

\newcommand{\spa}{\operatorname{span}}
\newcommand{\diam}{\operatorname{diam}}

\newcommand{\dist}{\operatorname{dist}}

\newcommand{\Fav}{\mathrm{Fav}}
\newcommand{\lip}{\mathrm{Lip}}
\newcommand{\clip}{\mathbf{C}_{\mathrm{lip}}}

\newcommand{\crofton}{2} %

\def\Barint_#1{\mathchoice
          {\mathop{\vrule width 6pt height 3 pt depth -2.5pt
                  \kern -8pt \intop}\nolimits_{#1}}%
          {\mathop{\vrule width 5pt height 3 pt depth -2.6pt
                  \kern -6pt \intop}\nolimits_{#1}}%
          {\mathop{\vrule width 5pt height 3 pt depth -2.6pt
                  \kern -6pt \intop}\nolimits_{#1}}%
          {\mathop{\vrule width 5pt height 3 pt depth -2.6pt
                  \kern -6pt \intop}\nolimits_{#1}}}

\numberwithin{equation}{section}

\theoremstyle{plain}
\newtheorem{thm}[equation]{Theorem}
\newtheorem*{"thm"}{"Theorem"}

\newtheorem{lemma}[equation]{Lemma}

\newtheorem{proposition}[equation]{Proposition}

\theoremstyle{definition}

\theoremstyle{remark}
\newtheorem{remark}[equation]{Remark}

\newcommand{\nref}[1]{(\hyperref[#1]{#1})}

\DeclareMathSymbol{\intop}  {\mathop}{mathx}{"B3}

\usepackage{xcolor}

\begin{document} 

\begin{abstract} Let $E \subset B(1) \subset \R^{2}$ be an $\mathcal{H}^{1}$ measurable set with $\mathcal{H}^{1}(E) < \infty$, and let $L \subset \R^{2}$ be a line segment with $\mathcal{H}^{1}(L) = \mathcal{H}^{1}(E)$. It is not hard to see that $\Fav(E) \leq \Fav(L)$. We prove that in the case of near equality, that is,
\begin{displaymath} \mathrm{Fav}(E) \geq \mathrm{Fav}(L) - \delta, \end{displaymath}
the set $E$ can be covered by an $\epsilon$-Lipschitz graph, up to a set of length $\epsilon$. The dependence between $\epsilon$ and $\delta$ is polynomial: in fact, the conclusions hold with $\epsilon = C\delta^{1/70}$ for an absolute constant $C > 0$. \end{abstract}

\maketitle

\tableofcontents

\section{Introduction}

Let $E \subset \R^{2}$ be $\mathcal{H}^{1}$ measurable with $\mathcal{H}^{1}(E) < \infty$. We recall the definition of Favard length:
\begin{displaymath} \Fav(E) = \int_{0}^{\pi} \mathcal{H}^{1}(\pi_{\theta}(E)) \, d\theta. \end{displaymath}
Here $\pi_{\theta} \colon \R^{2} \to \R$ is the orthogonal projection $\pi_{\theta}(x) = x \cdot (\cos \theta,\sin \theta)$. The definition of $\Fav(E)$ can be posed without the assumption $\mathcal{H}^{1}(E) < \infty$, but this hypothesis will be crucial for most of the statements below, and it will be assumed unless otherwise stated. A fundamental result in geometric measure theory is the Besicovitch projection theorem \cite{MR1513231} which relates Favard length and rectifiability: $\Fav(E) > 0$ if and only if $\mathcal{H}^{1}(E \cap \Gamma) > 0$ for some Lipschitz graph $\Gamma \subset \R^{2}$ -- in other words, $E$ is not purely $1$-unrectifiable.

The proof of the Besicovitch projection theorem is famous for being difficult to quantify, partly because of its reliance on the Lebesgue differentiation theorem: it is hard to decipher from the argument just how large the intersection $E \cap \Gamma$ is, and what the Lipschitz constant of $\Gamma$ is. In fact, it is non-trivial to even find the right question: for example, if $E \subset B(1)$, $\mathcal{H}^{1}(E) = 1$, and $\Fav(E) \geq \delta$ for some small but fixed constant $\delta > 0$, then it is not true that $\mathcal{H}^{1}(E \cap \Gamma) \geq \epsilon$ for some $\epsilon^{-1}$-Lipschitz graph $\Gamma \subset \R^{2}$, where $\epsilon = \epsilon(\delta) > 0$. We construct a relevant counterexample in Section \ref{sec:grid}.

In Theorem \ref{main}, we show that similar counterexamples are no longer possible if the assumption "$\Fav(E) \geq \delta$" is upgraded to "$\Fav(E) \geq 2\,\mathcal{H}^1(E) - \delta$" for a sufficiently small constant $\delta > 0$. The number "$2$" comes from the fact that $\Fav([0,1] \times \{0\}) = 2$ and that $[0,1] \times \{1\}$ has the maximal Favard length among sets of length unity (see \eqref{line-maximal-Favard}). 

\begin{thm}\label{main} For every $\epsilon > 0$ there exists $\delta > 0$ such that the following holds. Let $E \subset B(1)$ be an $\mathcal{H}^{1}$ measurable set with $\mathcal{H}^{1}(E) < \infty$, and assume that
\begin{equation}\label{form1} \Fav(E) \geq \Fav(L) - \delta, \end{equation}
where $L \subset \R^{2}$ is a line segment with $\mathcal{H}^{1}(L) = \mathcal{H}^{1}(E)$. Then, there exists an $\epsilon$-Lipschitz graph $\Gamma \subset \R^{2}$ such that $\mathcal{H}^{1}(E \cap \Gamma) \geq \mathcal{H}^{1}(E) - \epsilon$. One can take $\delta = \epsilon^{70}/{C}$ for an absolute constant ${C} > 1$.
\end{thm}

Thus, if $\Fav(E)$ is nearly maximal, the Besicovitch projection theorem can be quantified in a very strong way, whereas the example constructed in Section \ref{sec:grid} shows that any similar conclusion fails completely if we make the weaker assumption $\Fav(E) \geq \delta$. However, it remains plausible that the assumption $\Fav(E) \geq \delta$ is sufficient to guarantee a quantitative version of Besicovitch's theorem under the additional assumption that $E$ is $1$-Ahlfors regular, or satisfies other "multi-scale $1$-dimensionality" hypotheses. For recent partial results, and more discussion on this question, see \cite{2021arXiv210400826D,MO,MR4323640,MR2500864}. The problem is closely related to Vitushkin's conjecture \cite{MR0229838} on the connection between analytic capacity and Favard length, see \cite{2017arXiv171200594C,dabrowski2022analytic}. 

We briefly mention another closely related topic: if $E \subset \R^{2}$ is self-similar and purely $1$-unrectifiable, then $\Fav(E) = 0$ by the Besicovitch projection theorem. It is an interesting and very popular question to attempt quantifying the (sharp) rate of decay at which $\Fav(E_{n}) \to 0$, where $E_{n}$ is the "$n^{th}$ iteration" of the self-similar set. For recent developments, see \cite{MR2727621,MR3188064,MR2652491,MR3129103,2020arXiv200303620C,MR2740020,MR3526481,2022arXiv220207555L,MR2641082,MR1907902}. 

\subsection{Outline of the paper}
A quick outline of the article is as follows:  in Section \ref{s:prelim} we introduce Crofton's formula and prove that line segments maximise Favard length. In Section \ref{s:proof-main} we show how to prove Theorem \ref{main} using two main propositions, Proposition \ref{prop1} and Proposition \ref{prop2}. The former allows us to cover a set with almost maximal Favard length by a bounded number of Lipschitz graphs with small constant. The latter says that, in fact, there can only be one such graph. These two propositions are then proven in Section \ref{s:prop1} and Section \ref{s:prop2}, respectively. 
Section \ref{sec:grid} contains the counterexample mentioned above Theorem \ref{main}. Finally, in
Appendix \ref{section:lines-spanned-by-rectifiable-curves} we give an exact formula for the measure of lines spanned by two rectifiable curves - this is used in Section \ref{s:prop2} but it might be of independent interest.

\subsection*{Acknowledgements} The paper was written while the authors were visiting the Hausdorff Research Institute for Mathematics in Bonn during the research trimester \emph{Interactions between Geometric measure theory, Singular integrals, and PDE}. We would like to thank the institute and its staff for creating this opportunity for collaboration.

\section{Measure theoretic preliminaries}\label{s:prelim}

\subsection{Notation} For $x \in \R^{d}$ and $r > 0$, the notation $B(x,r)$ stands for a closed ball of radius $r$ centred at $x$. For $A \subset \R^{d}$, we denote the cardinality of $A$ by $\#A$, and we write $A(r) := \{x \in \R^{d} : \dist(x,A) \leq r\}$, where "$\dist$" is Euclidean distance. For $f,g \geq 0$, we write $f \lesssim g$ if there exists an absolute constant $C > 0$ such that $f \leq Cg$. The notation $f \gtrsim g$ means the same as $g \lesssim f$, and $f \sim g$ is shorthand for $f \lesssim g \lesssim f$. If the constant $C > 0$ is allowed to depend on some parameter "$p$", we signify this by writing $f \lesssim_{p} g$.

\subsection{Integralgeometry and Crofton's formula}
One of the main tools is Crofton's formula for rectifiable sets, which states the following. If $E \subset \R^{2}$ is $\mathcal{H}^{1}$ measurable $1$-rectifiable set with $\mathcal{H}^{1}(E) < \infty$, then
\begin{equation}\label{crofton0} \mathcal{H}^{1}(E) = \frac{1}{2}\,\int_{0}^{\pi} \int_{\R} \#(E \cap \pi_{\theta}^{-1}\{t\}) \, dt \, d\theta. \end{equation}
The equation \eqref{crofton0} is false without the rectifiability assumption, but the inequality "$\geq$" remains valid in this case. This formula (and the inequality) is a special case of a more general relation between Hausdorff measure and integralgeometric measure for $n$-rectifiable sets in $\R^{d}$, see Federer's paper \cite[Theorem 9.7]{MR22594}, or \cite[Theorem 3.2.26]{federer}. We next rephrase the formula \eqref{crofton0} in slightly more abstract terms. We define the following measure $\eta$ on the family $\mathcal{A} := \mathcal{A}(2,1)$ of all affine lines in $\R^{2}$:
\begin{displaymath} \eta(\mathcal{L}) = \int_{0}^{\pi} \mathcal{H}^{1}(\{t \in \R : \pi_{\theta}^{-1}\{t\} \in \mathcal{L}\}) \, d\theta, \qquad \mathcal{L} \subset \mathcal{A}. \end{displaymath}
With this notation, the Crofton formula \eqref{crofton0} can be rewritten as
\begin{equation}\label{crofton} \mathcal{H}^{1}(E) = \frac{1}{2}\,\int_{\mathcal{L}(E)} \#(E \cap \ell) \, d\eta(\ell), \end{equation}
where
\begin{equation*}
    \mathcal{L}(E) := \{\ell \in \mathcal{A} : E \cap \ell \neq \emptyset\}.
\end{equation*}

\begin{lemma}[The line segment maximizes Favard length]
\label{lemma:line-maximal-Favard}
If $E \subset \R^{2}$ is $\mathcal{H}^{1}$ measurable, $\mathcal{H}^{1}(E) < \infty$, and $L \subset \R^{2}$ is a line segment with $\mathcal{H}^{1}(E) = \mathcal{H}^{1}(L)$, then 
\begin{equation}\label{line-maximal-Favard} 
\Fav(E) \leq \Fav(L) 
\end{equation}
and 
\begin{equation}\label{Favard-defect} 
\Fav(L) - \Fav(E) 
\geq
\int_{\mathcal{L}(E)} \#(E \cap \ell) - 1 \, d\eta(\ell).
\end{equation}
If $E$ is rectifiable, then equality holds in \eqref{Favard-defect}.
\end{lemma}

\begin{proof}
Suppose $E \subset \R^{2}$ is $\mathcal{H}^{1}$ measurable, $\mathcal{H}^{1}(E) < \infty$, and $L \subset \R^{2}$ is a line segment with $\mathcal{H}^{1}(E) = \mathcal{H}^{1}(L)$. Then
\begin{equation} 
\label{eq:fav-leq-H1}
\Fav(E) = \eta(\mathcal{L}(E)) = \int_{\mathcal{L}(E)} 1 \, d\eta(\ell) \leq \int_{\mathcal{L}(E)} \#(E \cap \ell) \, d\eta(\ell) \leq \crofton\mathcal{H}^{1}(E).
\end{equation}
If we replace $E$ with the line segment $L$, then equality holds in both inequalities above. Thus,
$\Fav(L)
=
\crofton\mathcal{H}^{1}(L)
=
\crofton\mathcal{H}^{1}(E)
$, which combined with \eqref{eq:fav-leq-H1} (for $E$) proves \eqref{Favard-defect}. 

Next, \eqref{line-maximal-Favard} follows from the fact that the right-hand side of \eqref{Favard-defect} is nonnegative. Finally, if $E$ is rectifiable, then the second inequality in \eqref{eq:fav-leq-H1} becomes an equality, which implies that equality holds in \eqref{Favard-defect}.
\end{proof}

\subsection{Coarea formula}

We then record another tool in the proof of Theorem \ref{main}. It is closely related to Crofton's formula, but only considers the intersections with lines in a fixed direction. The price to pay is that the tangent of the rectifiable set enters the formula. It is a generalization of the following standard fact: If $f : [a,b] \to \R$ is $\alpha$-Lipschitz, then
\[
\mathcal{H}^1(\{(t,f(t)) : t \in [a,b]\}) = \int_a^b \sqrt{1+f'(t)^2} \, dt
\leq 
\sqrt{1+\alpha^2}\ (b-a).
\]

\begin{lemma}[Coarea formula] Let $\alpha > 0$. Let $E \subset \R^2$ be a countable union of $\alpha$-Lipschitz graphs over the $x$-axis. %
Then,
\begin{equation}\label{area} \mathcal{H}^{1}(A) 
\leq 
\sqrt{1+\alpha^2} 
\int_{\mathbb R} \#(A \cap \pi_0^{-1}\{t\}) \, dt 
\end{equation}
for all $\mathcal{H}^{1}$ measurable subsets $A \subset E$. (Recall that $\pi_0 : \R^2 \to \R$ is the projection onto the $x$-axis.)
\end{lemma}

\begin{proof}
This follows from the coarea formula for rectifiable sets. (See, e.g., \cite[Theorem 3.2.22]{federer} or \cite[Theorem 5.4.9]{Krantz-Parks}.)
\end{proof}

\section{Proof of Theorem \ref{main} in two main steps}\label{s:proof-main}
In this section we prove our main result using Proposition \ref{prop1} and Proposition \ref{prop2} introduced below. The former says that we can cover all of $E$, save for a tiny exceptional set, by a union of boundedly many Lipschitz graphs with small constant. The latter says that, in fact, there can be only one Lipschitz graph with small constant covering most of $E$, otherwise we run into contradiction with the assumption of almost maximal Favard length.

\subsection{Step 1. First reductions}
Let $E \subset \R^{2}$ be a Borel set with $\mathcal{H}^{1}(E) < \infty$. We start with the following simple lemma. 
 \begin{lemma}
\label{lemma:reduction-finite-disjoint-C1-curves}
It suffices to prove Theorem \ref{main} under the additional assumption that $E$ is a finite union of disjoint $C^1$ curves.
\end{lemma}
\begin{proof}
We may assume that $E \subset B(1)$ is rectifiable, because by the Besicovitch projection theorem, the rectifiable part of $E$ continues to satisfy all the assumptions of Theorem \ref{main} (with the same constant $\delta > 0$). By this assumption, $\mathcal{H}^{1}$ almost all of $E$ can be covered by a countable union of $C^{1}$-curves. Decomposing the curves further, we may assume that they are disjoint, and for any given $\eta > 0$ we may write
\begin{displaymath} E = \bigcup_{j = 1}^{M_1} (\gamma_{j} \cap E) \cup S, \end{displaymath}
where $\mathcal{H}^{1}(S) \leq \eta$, and $\mathcal{H}^{1}(E \cap \gamma_{j}) \geq (1 - \eta)\mathcal{H}^{1}(\gamma_{j})$. Now, the set $\bar{E} := \bigcup_{j = 1}^{M_1} \gamma_{j}$ satisfies 
\begin{displaymath} \mathcal{H}^{1}(\bar{E}) \leq (1 - \eta)^{-1}\mathcal{H}^{1}(E) \quad \text{and} \quad \Fav(\bar{E}) \geq \Fav(E) - \eta, \end{displaymath}
and is additionally a finite union of disjoint $C^{1}$-curves. If Theorem \ref{main} is already known under this additional assumption, we may now infer that $\mathcal{H}^{1}(\bar{E} \, \setminus \, \Gamma) \leq \epsilon$, where $\Gamma$ is an $\epsilon$-Lipschitz graph. But then also $\mathcal{H}^{1}(E \, \setminus \, \Gamma) \leq \mathcal{H}^{1}(E \, \setminus \, \bar{E}) + \mathcal{H}^{1}(\bar{E} \, \setminus \, \Gamma) \leq \eta+\epsilon$, and Theorem \ref{main} follows for $E$ by choosing the parameters $\epsilon,\eta$ appropriately.
\end{proof}

\subsection{Step 2. Minigraphs and how to merge them} By Lemma \ref{lemma:reduction-finite-disjoint-C1-curves}, in the sequel we may assume that $E$ is a finite union of disjoint $C^{1}$-curves $\gamma_{1},\ldots,\gamma_{M_1}$.  We further chop up each curve $\gamma_{j}$ into connected pieces whose tangent varies by less than $\alpha$, where $\alpha$ is a small constant depending on $\epsilon$ fixed later on (see \eqref{eq:alphadef}). At this point, we have managed to write $E$ as a finite union of disjoint $\alpha$-Lipschitz graphs $\gamma_{1},\ldots,\gamma_{M_1'}$, where $M_1 \leq M_1' < + \infty$. Each of these graphs will be called a \emph{minigraph}, and their collection is denoted $\mathcal{E}$. The main task in Theorem \ref{main} is to combine the minigraphs into bigger graphs.

To begin with, each of the minigraphs is an $\alpha$-Lipschitz graph over some line of the form
\begin{displaymath} \spa(\cos(k\pi /M_2),\sin(k\pi/M_2)), \qquad 0 \leq k \leq M_2 \sim \alpha^{-1}. \end{displaymath}
The vector $v_{k} := (\cos(k\pi /M_2),\sin(k\pi/M_2))$ will be called the \emph{direction} of the minigraph (if there are several suitable vectors for one minigraph, fix any one of them; we will only need to know that each minigraph is an $\alpha$-Lipschitz graph over the line spanned by its direction). Statements about the (relative) angles of minigraphs should always be interpreted as statements about the relative angles of  the direction vectors $v_k$. 

For $k \in \{0,\ldots,M_2\}$ fixed, we write $\mathcal{E}_{k} \subset \mathcal{E}$ for the subset of minigraphs with direction $v_{k}$. We suggest that the reader visualise the minigraphs as line segments $I$ with $\angle(I,\spa(v_k)) \leq \alpha$.  It seems likely that Theorem \ref{main} could be reduced to the case where $E$ is a finite union of line segments, but employing the minigraphs seems to spare us some unnecessary steps.

We write $E_{k} := \cup \mathcal{E}_{k}$. Thus 
\begin{equation}\label{e:def-E_j}
E = E_{0} \cup \ldots \cup E_{M_2}.
\end{equation}
It turns out that, except for a small error, each set $E_{k}$ is covered by a single Lipschitz graph with constant $\sim \alpha$ over $\spa(v_k)$. Indeed, note that Lemma \ref{lemma:line-maximal-Favard} and \eqref{form1} together imply
$  
\int_{\mathcal{L}(E)} \#(E \cap \ell) - 1 \, d\eta(\ell) \leq \delta. 
$
Then we have the following proposition, whose proof will be carried out in Section \ref{s:prop1}.
\begin{proposition}\label{prop1} 
There exist absolute constants $C_0, \alpha_0\in (0,1)$ and $\mathbf{C}_{\mathrm{lip}} > 1$ such that the following holds. Let $\delta,\epsilon\in (0,1)$ and $\alpha \in (0,\alpha_{0})$ be such that $\delta \le C_0 \alpha^{3}\epsilon^{2}$. Let $E \subset B(1)$ be a set with $\mathcal{H}^{1}(E) < \infty$ of the form 
\begin{equation*}
    E = \bigcup_{\gamma\in\mathcal{E}} \gamma,
\end{equation*}
where $\mathcal{E}$ is a finite collection of disjoint $\alpha$-Lipschitz graphs over a fixed line $L \subset \R^{2}$.  Assume further that $E$ satisfies  
\begin{equation}\label{form5}  
\int_{\mathcal{L}(E)} \#(E \cap \ell) - 1 \, d\eta(\ell) \leq \delta. 
\end{equation} 
Then, there exists a $\mathbf{C}_{\mathrm{lip}}\alpha$-Lipschitz graph $\Gamma$ over $L$, such that 
\begin{displaymath} \mathcal{H}^{1}(E \, \setminus \, \Gamma) \leq \epsilon. \end{displaymath}
\end{proposition}

\subsection{Step 3. There can only be one graph}\label{ss:only-one-graph}
In Proposition \ref{prop1} we managed to pack a majority of each set $E_{j}$ (as defined in \eqref{e:def-E_j}) to a Lipschitz graph of constant $\sim \alpha$, up to errors which tend to zero as $\delta \to 0$ in the main assumption \eqref{form1}. However, at this point there might be up to $\sim \alpha^{-1}$ distinct Lipschitz graphs, and to prove Theorem \ref{main}, we would (roughly speaking) like to reduce their number to one. That this should be possible is not hard to believe: if $E$ consists of several distinct Lipschitz graphs of substantial measure, which nevertheless cannot be fit into a single Lipschitz graph, then $\Fav(E)$ cannot possibly be maximal.\\

\noindent
We turn to the details. 
We recall the "given" constant $\epsilon > 0$ from the statement of Theorem \ref{main}, and we set \begin{equation*}
    \delta := \frac{\epsilon^{70}}{\mathbf{C}_{\mathrm{thm}}}
\end{equation*}
for a sufficiently large absolute constant $\mathbf{C}_{\mathrm{thm}} > 1$. We define also
\begin{equation}\label{eq:alphadef}
\alpha := \bigg(\frac{\epsilon}{\mathbf{C}_{\mathrm{alp}}}\bigg)^{10} 
\end{equation}
for some universal $\mathbf{C}_{\mathrm{alp}}>1$. The universal constant $\mathbf{C}_{\mathrm{thm}}$ will depend on $\mathbf{C}_{\mathrm{alp}}$, whereas $\mathbf{C}_{\mathrm{alp}}$ depends only on $\mathbf{C}_{\mathrm{lip}}$ and another constant $\mathbf{C}_{\mathrm{sep}}$, which is introduced below. We record that
\begin{equation}\label{tform1} 
    \alpha^{7} = \mathbf{C}_{\mathrm{alp}}^{-70}\epsilon^{70} = \mathbf{C}_{\mathrm{thm}} \mathbf{C}_{\mathrm{alp}}^{-70} \cdot \delta. 
\end{equation}

Recall, once more, the decompositions $\mathcal{E} = \mathcal{E}_{0} \cup \ldots \cup \mathcal{E}_{M_2}$ and $E = E_{0} \cup \ldots \cup E_{M_2}$ from the previous subsection: this decomposition depends on the parameter $\alpha$ fixed above.
In addition to the decomposition $E = E_{0} \cup \ldots \cup E_{M_2}$, we will also need another, coarser, decomposition of $E$ in this section. Write $\kappa := \tfrac{1}{10}$, fix $M_3 \sim \alpha^{-\kappa}$, and decompose $\mathcal{E} = \mathcal{F}_{0} \cup \ldots \cup \mathcal{F}_{M_3}$ in such a way that 
\begin{itemize}
\item each $\mathcal{F}_{k}$ is a union of finitely many consecutive families $\mathcal{E}_{j}$, and
\item $\mathcal{F}_{k}$ contains those minigraphs whose direction makes an angle $\leq \alpha^{\kappa}$ with $w_{k} = (\cos(k\pi/M_{3}),\sin(k\pi/M_{3}))$, for $0 \leq k \leq M_3$.
\end{itemize}
We write 
\begin{displaymath} F_{k} := \cup \mathcal{F}_{k}, \qquad 0 \leq k \leq M_3 \sim \alpha^{-\kappa}. \end{displaymath}
At this point, we consider two distinct cases. Let $\mathbf{C}_{\mathrm{sep}}$ be a large constant depending only on the absolute constant $\mathbf{C}_{\mathrm{lip}}$ appearing in Proposition \ref{prop1} (the letters "$\mathrm{sep}$" stand for "separation"). Thus, the constant $\mathbf{C}_{\mathrm{sep}}$ is also absolute, and we may (and will) assume that $\mathbf{C}_{\mathrm{alp}}$ is large relative to $\mathbf{C}_{\mathrm{sep}}$. \\

\textit{\underline{Case 1.}} Given the constant $\epsilon > 0$ from Theorem \ref{main}, the first case is that we can find consecutive sets $F_{k},F_{k + 1},\ldots,F_{k + \mathbf{C}_{\mathrm{sep}}}$ with the property
\begin{equation}\label{form22} \mathcal{H}^{1}(E \, \setminus \, (F_{k} \cup \ldots \cup F_{k + \mathbf{C}_{\mathrm{sep}}})) \leq \epsilon. \end{equation}
In this case we note that $F := F_{k} \cup \ldots \cup F_{k + \mathbf{C}_{\mathrm{sep}}}$ is a union of minigraphs whose directions are within $\lesssim \mathbf{C}_{\mathrm{sep}} \alpha^{\kappa}$ of the fixed vector $w_{k}$. In particular, $F$ can be expressed as a union of finitely many disjoint $\alpha_0$-Lipschitz graphs over the line $\mathrm{span}(w_k)$, with $\alpha_0\sim \mathbf{C}_{\mathrm{sep}} \alpha^{\kappa}$. This will place us in a positions to use Proposition \ref{prop1} (with $E$ replaced by $F$ and $\alpha$ replaced by $\alpha_0$). Of course also
\begin{displaymath} \int_{\mathcal{L}(F)} \#(F \cap \ell) - 1 \, d\eta(\ell) \leq \int_{\mathcal{L}(E)} \#(E \cap \ell) -1\, d\eta(\ell)\leq \delta, \end{displaymath} 
so the analogue of the assumption \eqref{form5} is valid for $F$ in place of $E$. We also note that
\begin{displaymath} \delta = \mathbf{C}_{\mathrm{thm}}^{-1}\epsilon^{70} \leq \mathbf{C}_{\mathrm{thm}}^{-1}\mathbf{C}_{\mathrm{alp}}^{3} \cdot (\epsilon/\mathbf{C}_{\mathrm{alp}})^{3} \cdot \epsilon^{2} = (\mathbf{C}_{\mathrm{thm}}^{-1}\mathbf{C}_{\mathrm{alp}}^{3}) \cdot \alpha^{3\kappa}\epsilon^{2} \sim (\mathbf{C}_{\mathrm{thm}}^{-1}\mathbf{C}_{\mathrm{alp}}^{3}\mathbf{C}_{\mathrm{sep}}^{-3}) \cdot \alpha_0^{3}\epsilon^{2} , \end{displaymath} 
so if $\mathbf{C}_{\mathrm{thm}}$ is sufficiently large relative to $\mathbf{C}_{\mathrm{alp}}$, then the hypothesis in Proposition \ref{prop1} on the relation between $\delta,\alpha_0$, and $\epsilon$ is satisfied (the constant $\mathbf{C}_{\mathrm{sep}}$ is large, so it can be safely ignored here). Consequently, there exists a Lipschitz graph $\Gamma \subset \R^{2}$ of constant $\lesssim \mathbf{C}_{\mathrm{lip}}\mathbf{C}_{\mathrm{sep}} \cdot \alpha^{\kappa} = \mathbf{C}_{\mathrm{lip}}\mathbf{C}_{\mathrm{sep}} \cdot \epsilon/\mathbf{C}_{\mathrm{alp}}$  with the property
\begin{displaymath} \mathcal{H}^{1}(F \, \setminus \, \Gamma) \leq \epsilon, \end{displaymath}
and consequently $\mathcal{H}^{1}(E \, \setminus \, \Gamma) \leq 2\epsilon$. By choosing $\mathbf{C}_{\mathrm{alp}}$ sufficiently large relative to $\mathbf{C}_{\mathrm{sep}}$ and $\mathbf{C}_{\mathrm{lip}}$, we may ensure that $\Gamma$ is an $\epsilon$-Lipschitz graph, as desired. \\

\textit{\underline{Case 2.}} We then move to consider the other option, where $E$ cannot be exhausted, up to measure $\epsilon$, by a constant number of consecutive sets $F_{k},F_{k + 1},\ldots,F_{k + \mathbf{C}_{\mathrm{sep}}}$. 
Since \eqref{form22} fails for every $k$, we may find an index pair $k,l \in \{0,\ldots,M_3\}$ with $|k - l| \geq \mathbf{C}_{\mathrm{sep}}$ such that 
\begin{equation}\label{form19} 
        \mathcal{H}^{1}(F_{k})   \geq \alpha^{2\kappa} \quad \text{and} \quad \mathcal{H}^{1}(F_{l}) \geq  \alpha^{2\kappa}. 
\end{equation}
This follows immediately from the pigeonhole principle, recalling that the cardinality of the pieces $F_{k}$ is $\lesssim \alpha^{-\kappa}$, and also that $\alpha^{\kappa}$ is much smaller than $\epsilon$ by \eqref{eq:alphadef}. 

\begin{remark}
Recall that the "separation" constant $\mathbf{C}_{\mathrm{sep}}$ above has been chosen to be large relative to the constant $\mathbf{C}_{\mathrm{lip}}$ in Proposition \ref{prop1}: morally, if $\Gamma_{1},\Gamma_{2}$ are two $\mathbf{C}_{\mathrm{lip}}\alpha^{\kappa}$-Lipschitz graphs over lines $L_{1},L_{2}$ with $\angle(L_{1},L_{2}) \geq \mathbf{C}_{\mathrm{sep}}\alpha^{\kappa}$, we need to know that $\Gamma_{1}$ and $\Gamma_{2}$ are still "transversal" (their tangents form angles $\geq \tfrac{1}{2}\mathbf{C}_{\mathrm{sep}}\alpha^{\kappa}$ with each other).
\end{remark}
The next key proposition will imply that Case 2 cannot happen:

\begin{proposition}\label{prop2} Suppose that $\mathbf{C}_{\mathrm{sep}} > 0$ is sufficiently large, and suppose that there are $k,l \in \{0,\ldots,M_3\}$ with $|k-l|\geq \mathbf{C}_{\mathrm{sep}}$ such that 
        \begin{equation*}
                \mathcal{H}^{1}(F_{k}) \geq \alpha^{2\kappa} \quad \text{and} \quad \mathcal{H}^{1}(F_{l}) \geq \alpha^{2\kappa}. 
        \end{equation*}
        Then 
        \begin{equation}\label{tform2} \int_{\mathcal{L}(E)} \#(E \cap \ell) - 1 \, d\eta(\ell) \gtrsim \alpha^{7}. \end{equation}
        \end{proposition}
As we recorded in \eqref{tform1}, we have $\alpha^{7} = \mathbf{C}_{\mathrm{thm}}\mathbf{C}_{\mathrm{alp}}^{-70} \cdot \delta$. Thus, if $\mathbf{C}_{\mathrm{thm}}$ is chosen sufficiently large relative to $\mathbf{C}_{\mathrm{alp}}$ and the implicit absolute constants in \eqref{tform2}, then \eqref{tform2} would lead to the contradiction
\begin{displaymath} \delta \geq \int_{\mathcal{L}(E)} \#(E \cap \ell) - 1 \, d\eta(\ell) > \delta. \end{displaymath}   
(For the first inequality, recall \eqref{Favard-defect} and our main assumption \eqref{form1}.) Thus, with the choices of constants specified in this section, Case 2 cannot occur. This concludes the proof of Theorem \ref{main}. 

In the next two sections we prove the two key results used above, Propositions \ref{prop1} and \ref{prop2}.

\section{Proof of Proposition \ref{prop1}}\label{s:prop1}

 Let $E \subset \R^{2}$ be as in the proposition. With no loss of generality, we may assume that $L$ is the $x$-axis, so the minigraphs in $\mathcal{E}$ are roughly horizontal.
We introduce further notation. We write 
\begin{displaymath} \mathcal{C}_{\beta} := \{(x,y) \in \R^{2} : |y| \geq \beta |x|\}, \qquad \beta > 0. \end{displaymath}
Thus, the smaller the $\beta$, the wider the cone. We also write 
\begin{displaymath} \mathcal{C}_{\beta}(x) := x + \mathcal{C}_{\beta} \quad \text{and} \quad \mathcal{C}_{\beta}(x,r) := \mathcal{C}_{\beta}(x) \cap B(x,r). \end{displaymath}
With this notation, if a set $\Gamma \subset \R^{2}$ satisfies $\Gamma \cap \mathcal{C}_{\beta}(x) = \{x\}$ for all $x \in \Gamma$, then $\Gamma$ is (a subset of) a $\beta$-Lipschitz graph. Thus, in view of Proposition \ref{prop1}, it would be desirable to show that $E \cap \mathcal{C}_{\clip\alpha}(x) = \{x\}$ for all $x \in E$. In reality, we will prove a similar statement about a subset of $E$ (of nearly full length). It is worth noting that a toy version of these statements is already present in our hypotheses: each minigraph $\gamma \in \mathcal{E}$ is an $\alpha$-Lipschitz graph over the $x$-axis.

\newcommand{\maxconical}[3]{\Theta^*_{#1,#2}(#3)}

Define the maximal conical density
\[
\maxconical{E}{\beta}{x}
=
\sup_{r > 0} \frac{\mathcal{H}^{1}(\mathcal{C}_{\beta}(x,r) \cap E)}{r}.
\]
Lemma \ref{lemma1} says that points of high conical density are negligible, whereas Lemma \ref{lemma2} says that points of low conical density can be mostly contained in a Lipschitz graph.

\begin{lemma}[High conical density points are negligible]\label{lemma1} 
Let $E\subset B(1),\ \alpha\in(0,\alpha_0)$ and $\delta\in (0,1)$ be as in Proposition \ref{prop1}, so that in particular \eqref{form5} holds. Let $\varepsilon > 0$. If the absolute constant $\clip > 0$ is chosen sufficiently large, then
\begin{equation}\label{form9} \mathcal{H}^{1}(\{x \in E : \maxconical{E}{\alpha'}{x}
 \geq \varepsilon \}) \lesssim \frac{\delta}{\varepsilon \alpha^{2}}, \end{equation}
where $\alpha'\coloneqq\clip\alpha/2$. \end{lemma}

Write $\ell_{x,\theta} := \pi_{\theta}^{-1}\{\pi_{\theta}(x)\}$ for $\theta \in [0, \pi)$, so that $\ell_{0,\theta} = \spa(\cos\theta, \sin\theta)^\perp$. Let $J(\beta) \subset [0,\pi)$ be the set of directions in the cone $\mathcal{C}_{\beta}$, i.e., 
\begin{displaymath} J(\beta) = \{\theta \in [0,\pi) : \ell_{0,\theta} \subset \mathcal{C}_{\beta}\} = \{\theta \in [0,\pi) : \spa(\cos\theta, \sin\theta)^\perp \subset \mathcal{C}_{\beta}\}. \end{displaymath}
If $\ell$ is a line, we let $\ell(w)$ denote the tube that is the $w$-neighborhood of $\ell$. For a tube $T = \ell(w)$, we denote $w(T) = w$.

To prove Lemma \ref{lemma1}, we use the Besicovitch alternative:

\begin{lemma}[The Besicovitch alternative]
\label{lemma:besicovitch-alternative}
Let $E \subset \R^2$ and $\beta \leq 1$. Then for all $x \in E$ and $H \geq 1$, at least one of the following two alternatives holds:
 \begin{enumerate}[label=(A\arabic*)]
 \item\label{A1} There exists a set $I_{x} \subset J(\beta)$ of measure $\mathcal{H}^{1}(I_{x}) \geq H^{-1}$ such that 
 \begin{displaymath} \#(E \cap \ell_{x,\theta}) \geq 2, \qquad \theta \in I_{x}. \end{displaymath}
 \item\label{A2} There exists a set $J_{x} \subset J(\beta)$ of measure $\mathcal{H}^{1}(J_{x}) \gtrsim H^{-1}$ and the following property: for every $\theta \in J_{x}$, there is a tube $T = T_{x,\theta} = \ell_{x,\theta}(w(T))$ centred around $\ell_{x,\theta}$ such that
 \begin{displaymath} \mathcal{H}^{1}(E \cap T) \gtrsim \maxconical{E}{\beta}{x} \cdot H \cdot w(T). \end{displaymath} 
 \end{enumerate}
\end{lemma}

This alternative is part of Besicovitch's original argument \cite{MR1513231} for the Besicovitch projection theorem. For a more recent presentation, see \cite[p. 86-87]{MR867284}. We include the details for completeness.

\begin{proof}[Proof of Lemma \ref{lemma:besicovitch-alternative}]
Let $E,x,\beta,H$ be as in the statement of the lemma. Let $\varepsilon \coloneqq \maxconical{E}{\beta}{x}$, so that there exists an $r>0$ such that $\mathcal{H}^{1}(\mathcal{C}_{\beta}(x,r) \cap E) \geq \varepsilon r$. We set also $J\coloneqq J(\beta).$ 

If the alternative (A1) fails, then
\begin{displaymath} \mathcal{H}^{1}(\{\theta \in J : \#(\mathcal{C}_{\beta}(x,r) \cap E \cap \ell_{x,\theta}) \geq 2\}) \leq H^{-1}. \end{displaymath}
 Since evidently $x \in \mathcal{C}_{\beta}(x,r) \cap E \cap \ell_{x,\theta}$, this implies that most of the lines $\ell_{x,\theta}$ do not intersect the set $\mathcal{C}_{\beta}(x,r) \cap E$ outside $x$. Consequently, $\mathcal{C}_{\beta}(x,r) \cap E$ is contained in a union of narrow cones $\mathcal{C}_{1},\mathcal{C}_{2},\ldots$ which are centred around certain lines $\ell_{x,\theta_{j}}$ with $\theta_{j} \in J$, and whose opening angles $\beta_{1},\beta_{2},\ldots$ satisfy $\sum \beta_{j} \leq 2H^{-1}$. We may arrange that the cones have the form
 \begin{displaymath} \mathcal{C}_{j} \coloneqq \mathcal{C}(I_{j}) \coloneqq \cup \{\ell_{x,\theta} : \theta \in I_{j}\}, \end{displaymath}
 where $I_{j} \subset J$ is a dyadic interval, $|I_{j}| = \beta_{j}$, and $\theta_{j} \in J$ is the midpoint of $I_{j}$. We may also assume that the dyadic intervals $I_{j}$ are disjoint, so the sets $\mathcal{C}_{j} \, \setminus \, \{x\}$ are disjoint. 
 
To use these cones to arrive at alternative (A2), recall that $\mathcal{H}^{1}(\mathcal{C}_{\beta}(x,r) \cap E) \geq \varepsilon r$, where $\varepsilon = \maxconical{E}{\beta}{x}$. Now, we throw away cones which are not \emph{heavy}: we call a cone \emph{heavy} if it satisfies
\begin{equation}\label{form36} \mathcal{H}^{1}(\mathcal{C}_{j} \cap B(x,r) \cap E) \geq \tfrac{1}{4} \cdot \varepsilon H |I_{j}| \cdot r. \end{equation}
The total length of $\mathcal{C}_{\beta}(x,r) \cap E$ contained in the non-heavy cones is bounded from above by 
\begin{displaymath} \frac{\varepsilon Hr}{4} \sum_{j\in\N} |I_{j}| \leq \frac{\varepsilon r}{2} \leq \tfrac{1}{2} \mathcal{H}^{1}(\mathcal{C}_{\beta}(x,r) \cap E), \end{displaymath}
so at least half of the length in $\mathcal{C}_{\beta}(x,r) \cap E$ is contained in the union of the heavy cones. In the sequel, we assume that all the cones $\mathcal{C}_{j}$ are heavy. 

Next, we would like to prove that $\sum \beta_{j} = \sum |I_{j}| \gtrsim H^{-1}$. This would be easy if the heavy cones also satisfied an upper bound roughly matching the lower bound in \eqref{form36}. If we knew this, then we could estimate
\begin{equation}\label{form37} \sum_{j \in \N} |I_{j}| \gtrsim (\varepsilon H r)^{-1} \sum_{j \in \N} \mathcal{H}^{1}(\mathcal{C}_{j} \cap B(x,r) \cap E) \gtrsim H^{-1}. \end{equation}
This desired upper bound in \eqref{form36} need not be true to begin with, but can be easily arranged. Fix a heavy cone $\mathcal{C}(I_{j})$, and  perform the following stopping time argument: the dyadic interval $I_{j}$ is successively replaced by its parent "$\hat{I}_{j}$" until either the upper bound 
\begin{equation}\label{form38} \mathcal{H}^{1}(\mathcal{C}(\hat{I}_{j}) \cap B(x,r) \cap E) \leq \varepsilon H|\hat{I}_{j}| \cdot r \end{equation}
holds, or then $\hat{I}_{j} = J$. This procedure gives rise to a new collection of cones $\mathcal{C}(\hat{I}_{j})$ which are evidently still heavy, and whose union covers the union of the initial heavy cones. Since the intervals $\hat{I}_{j}$ are dyadic, we may arrange that the new heavy cones are disjoint outside $\{x\}$ without violating the previous two properties. 

At this point, either $\hat{I}_{j} = J$ for some index $j$, in which case \eqref{form37} is trivially true (using $|J| \sim 1$), or then the upper bound \eqref{form38} holds for all the heavy cones. In this case the lower bound \eqref{form37} holds by the very calculation shown in \eqref{form37}.

We are now fully equipped to establish alternative (A2). Consider a line $\ell_{x,\theta}$ contained in the union of the heavy cones. According to \eqref{form37}, the set of angles $\theta \in J$ of such lines has length $\gtrsim H^{-1}$. This set of angles is the set $J_{x} \subset J$ whose existence is claimed in (A2). It remains to associate the tube $T_{x,\theta}$ to each line $\ell_{x,\theta}$ with $\theta \in J_{x}$. Let $\mathcal{C}(I_{j}) = \mathcal{C}_{j} \supset \ell_{x,\theta}$ be the (unique) heavy cone containing $\ell_{x,\theta}$. The opening angle of $\mathcal{C}_{j}$ is $\beta_{j} = |I_{j}| \in (0,|J|]$, and it follows by elementary geometry that 
\begin{displaymath} \mathcal{C}_{j} \cap B(x,r) \subset \ell_{x,\theta}(2\beta_{j}r) =: T_{x,\theta}. \end{displaymath}
Finally,
\begin{displaymath} \mathcal{H}^{1}(E \cap T_{x,\theta}) \geq \mathcal{H}^{1}(\mathcal{C}_{j} \cap B(x,r) \cap E) \gtrsim \varepsilon H \beta_{j} \cdot r \sim \varepsilon H \cdot w(T), \end{displaymath}
as claimed in alternative (A2).
\end{proof}

\begin{proof}[Proof of Lemma \ref{lemma1}] The main geometric observation is the following: every minigraph in $\mathcal{E}$ is an $\alpha^{-1}$-Lipschitz graph over every line $L_\theta := \spa(\cos \theta,\sin \theta) = \ell_{0,\theta}^\perp$ with $\theta \in J(\alpha')$ (recall that $\alpha'=\clip\alpha/2$). This is simply because the minigraphs in $\mathcal{E}$ are $\alpha$-Lipschitz graphs over the $x$-axis, but for all $\theta \in J(\alpha')$, the lines $L_\theta$ form an angle $\gtrsim \alpha$ with the $y$-axis. Thus, $E$ is a union of finitely many $\alpha^{-1}$-Lipschitz graphs over $L_{\theta}$, for every $\theta \in J(\alpha')$. This places us in a position to use the area formula \eqref{area}: for every $\theta \in J(\alpha')$ and every $\mathcal{H}^{1}$ measurable subset $E' \subset E$ we have
\begin{equation}\label{form7} \int_{\pi_{\theta}(E')} \#(E' \cap \pi_{\theta}^{-1}\{t\}) \, dt \gtrsim \alpha\mathcal{H}^{1}(E'). \end{equation} 
 
Let
\[
R =\{x \in E : \maxconical{E}{\alpha'}{x} \geq \varepsilon \}.
\]
Fix $H \geq 1$. (We will eventually choose $H \sim 1/(\alpha \varepsilon)$; see \eqref{form15} below.) By Lemma \ref{lemma:besicovitch-alternative} (with $\beta = \alpha'$), we can write $R = R_1 \cup R_2$, where alternative (A1) holds on $R_1$ and (A2) holds on $R_2$. To prove \eqref{form9}, it suffices to show
\begin{equation}\label{Ri-bound} 
\mathcal{H}^{1}(R_i) \lesssim \frac{\delta }{\varepsilon\alpha^2} \qquad \text{ for } i = 1,2. 
\end{equation}

We first consider $R_1$. Recall the sets $I_{x} \subset J(\alpha')$ defined in (A1). Since $E$ is a union of finitely many compact Lipschitz graphs, there are no measurability issues, and we may freely use Fubini's theorem:
 \begin{equation}\label{form8} H^{-1}\mathcal{H}^{1}(R_1) \leq \int_{R_1} \mathcal{H}^{1}(I_{x}) \, d\mathcal{H}^{1}(x) = \int_{J(\alpha')} \mathcal{H}^{1}(\{x \in R_1 : \theta \in I_{x}\}) \, d\theta. \end{equation}
 For $\theta \in J(\alpha')$ fixed, abbreviate $R_{\theta}' := \{x \in R_1 : \theta \in I_{x}\}$. Write also
\begin{displaymath} E_{\theta}' := \bigcup_{t \in \pi_{\theta}(R_{\theta}')} (E \cap \pi_{\theta}^{-1}\{t\}), \end{displaymath} 
so certainly $R_{\theta}' \subset E_{\theta}'$. Note that if $t \in \pi_{\theta}(E_{\theta}')$, then $t = \pi_{\theta}(x)$ for some $x \in R_{\theta}'$. Thus $\theta \in I_{x}$ by definition, so 
\begin{displaymath} \#(E_{\theta}' \cap \pi_{\theta}^{-1}\{t\}) = \#(E \cap \ell_{x,\theta}) \geq 2. \end{displaymath}
Therefore
 \begin{equation}\label{form39} \#(E_{\theta}' \cap \pi_{\theta}^{-1}\{t\}) - 1 \sim \#(E_{\theta}' \cap \pi_{\theta}^{-1}\{t\}), \qquad t \in \pi_{\theta}(E_{\theta}'). \end{equation}
We may now deduce from \eqref{form7} applied to $E' := E_{\theta}'$, and \eqref{form39}, that
\begin{displaymath} \int_{\pi_{\theta}(E_{\theta}')} \#(E_{\theta}' \cap \pi_{\theta}^{-1}\{t\}) - 1 \, dt \sim \int_{\pi_{\theta}(E_{\theta}')} \#(E_{\theta}' \cap \pi_{\theta}^{-1}\{t\}) \, dt \gtrsim \alpha \mathcal{H}^{1}(E_{\theta}') \geq \alpha\mathcal{H}^{1}(R_{\theta}'), \end{displaymath}
and finally
\begin{displaymath} \int_{\mathcal{L}(E)} \#(E \cap \ell) - 1 \, d\eta(\ell) \geq \int_{J(\alpha')} \int \#(E_{\theta}' \cap \pi_{\theta}^{-1}\{t\}) - 1 \, dt \, d\theta \stackrel{\eqref{form8}}{\geq} \alpha H^{-1}\mathcal{H}^{1}(R_1). \end{displaymath}
By \eqref{form5} the left hand side is bounded from above by $\delta$, so
\begin{equation}\label{form40} \mathcal{H}^{1}(R_1) \lesssim \frac{\delta H}{\alpha}. \end{equation}
Recalling that we promised to choose $H \sim 1/(\alpha \varepsilon)$ in the end, the bound above implies \eqref{Ri-bound} for $R_1$.

Next, we tackle $R_2$. This time we define $R_{\theta}' := \{x \in R_2 : \theta \in J_{x}\} \subset E$, and we deduce exactly as in \eqref{form8} that
\begin{equation}\label{form10} H^{-1}\mathcal{H}^{1}(R_2) \lesssim \int_{J(\alpha')} \mathcal{H}^{1}(R_{\theta}') \, d\theta. \end{equation}
Fix $\theta \in J(\alpha')$ with $R_{\theta}' \neq \emptyset$. For each $x \in R_{\theta}'$, by definition, there exists a tube $T = T_{x,\theta}$ centred around $\ell_{x,\theta}$ with the property 
\begin{equation}\label{form12} \mathcal{H}^{1}(E \cap T) \gtrsim \varepsilon H \cdot w(T). \end{equation}
The tubes $\{T_{x,\theta} : x \in R_{\theta}'\}$ may overlap, but they are all parallel. It follows from an application of the Besicovitch covering theorem (to the projections $I_{x,\theta} := \pi_{\theta}(T_{x,\theta}) \subset \R$) that there exists a countable sub-collection $\mathcal{T}_{\theta} \subset \{T_{x,\theta} : x \in R_{\theta}'\}$, with the properties
\begin{equation}\label{form11} R_{\theta}' \subset \bigcup_{x \in R_{\theta}'} T_{x,\theta} \subset \bigcup_{T \in \mathcal{T}_{\theta}} T \quad \text{and} \quad \sum_{T \in \mathcal{T}_{\theta}} \mathbf{1}_{T} \lesssim 1. \end{equation}
Fix $T \in \mathcal{T}_{\theta}$, and let $\textbf{Bad}(E \cap T) \subset  E \cap T$ consist of those points $x \in E \cap T$ with $\#(\ell_{x,\theta} \cap E) = 1$. We apply the coarea formula \eqref{area} to the set $A := \textbf{Bad}(E \cap T) \subset E$. Recalling that for every $\theta \in J(\alpha')$ the set $E$ is a union of finitely many $\alpha^{-1}$-Lipschitz graphs over $L_\theta$ (see remark above \eqref{form7}) we get that
\begin{equation}\label{form13} \mathcal{H}^{1}(\textbf{Bad}(E \cap T)) \lesssim \frac{1}{\alpha} \int_{\pi_{\theta}(T)} 1 \, dt = \frac{w(T)}{\alpha}. \end{equation}
Now, for a suitable choice $H \sim 1/(\alpha\varepsilon)$, a combination of \eqref{form12} and \eqref{form13} shows that
\begin{equation}\label{form15} \mathcal{H}^{1}((E \cap T) \, \setminus \, \mathbf{Bad}(E \cap T)) \geq \tfrac{1}{2}\mathcal{H}^{1}(E \cap T). \end{equation} 
At this point, we simplify notation by setting 
\begin{displaymath} E_{\theta} := \bigcup_{T \in \mathcal{T}_{\theta}} (E \cap T) \, \setminus \, \mathbf{Bad}(E \cap T) \subset E. \end{displaymath} 
By the definition of the sets $\mathbf{Bad}(E \cap T)$, if $x \in E_{\theta}$, then $\#(E \cap \ell_{x,\theta}) \geq 2$, and therefore 
 \begin{equation}\label{form14} \#(E \cap \pi_{\theta}^{-1}\{t\}) - 1 \sim \#(E \cap \pi_{\theta}^{-1}\{t\}) \ge \#(E_{\theta} \cap \pi_{\theta}^{-1}\{t\}), \qquad t \in \pi_{\theta}(E_{\theta}). \end{equation}
 It follows that
 \begin{align*} \int_{\mathcal{L}(E)} \#(E \cap \ell) - 1 \, d\eta(\ell) & \geq \int_{J(\alpha')} \int \#(E \cap \pi_{\theta}^{-1}\{t\}) - 1 \, dt \, d\theta\\
 &\stackrel{\eqref{form14}}{\gtrsim} \int_{J(\alpha')} \int_{\pi_{\theta}(E_{\theta})} \#(E_{\theta} \cap \pi_{\theta}^{-1}\{t\}) \, dt \, d\theta\\
 &\stackrel{\eqref{form11}}{\gtrsim} \int_{J(\alpha')} \sum_{T \in \mathcal{T}_{\theta}} \int_{\pi_{\theta}(E_{\theta} \cap T)} \#(E_{\theta} \cap \pi_{\theta}^{-1}\{t\}) \, dt \, d\theta\\
 &\stackrel{\eqref{form7}}{\gtrsim} \alpha \int_{J(\alpha')} \sum_{T \in \mathcal{T}_{\theta}} \mathcal{H}^{1}(E_{\theta} \cap T) \, d\theta\\
 &\stackrel{\eqref{form15}}{\geq} \frac{\alpha}{2} \int_{J(\alpha')} \sum_{T \in \mathcal{T}_{\theta}} \mathcal{H}^{1}(E \cap T) \, d\theta\\
 &\stackrel{\eqref{form11}}{\geq} \alpha \int_{J(\alpha')} \mathcal{H}^{1}(R_{\theta}') \, d\theta \stackrel{\eqref{form10}}{\geq} \frac{\alpha}{H} \cdot \mathcal{H}^{1}(R_2). \end{align*} 
 Recalling once again from \eqref{form5} that the left hand side above is $\leq \delta$, we deduce that
 \begin{displaymath} \mathcal{H}^{1}(R_2) \lesssim \frac{\delta H}{\alpha} \sim \frac{\delta}{\varepsilon \alpha^{2}}, \end{displaymath}
 which is \eqref{Ri-bound} for $R_2$. The proof of Lemma \ref{lemma1} is complete.  \end{proof}

Next, repeating the classical "two cones" argument of Besicovitch, we show that we can pack most of points of low conical density into a single Lipschitz graph. 

\begin{lemma}[Most low conical density points fit into a Lipschitz graph]
\label{lemma2} 

Let $E \subset B(1) \subset \R^{2}$ and let $\varepsilon\in(0,1),\, \beta \in (0,\frac{1}{2})$. Then, there exists a $2\beta$-Lipschitz graph $\Gamma \subset \R^{2}$ over the $x$-axis such that 
\begin{displaymath} \mathcal{H}^{1}(\{x \in E : \maxconical{E}{\beta}{x} \leq \varepsilon\} \, \setminus \, \Gamma) \lesssim \varepsilon/\beta. \end{displaymath}
\end{lemma}

\begin{proof} Let $G = \{x \in E : \maxconical{E}{\beta}{x} \leq \varepsilon\}$. Our task is to find a subset $\Gamma \subset G$ with $\mathcal H^1(G \, \setminus \, \Gamma) \lesssim \varepsilon/\beta$ and the property $\mathcal{C}_{2\beta}(x) \cap \Gamma = \{x\}$ for all $x \in \Gamma$. Then $\Gamma$ extends to a $2\beta$-Lipschitz graph, as desired. 

Let $B$ be the set of points $x \in G$ with the ``bad'' property that there exists a point $y \in G \cap \mathcal{C}_{2\beta}(x)$ with $y \neq x$. The goal is to show that $\mathcal{H}^{1}(B) \lesssim \varepsilon/\beta$. For each $x \in B$, let $r(x) = \sup \{|x-y| : y \in G \cap \mathcal{C}_{2\beta}(x)\}$, so
\begin{equation}\label{form16} B \cap \mathcal{C}_{2\beta}(x) \subset B(x,r(x)), \qquad x \in B. \end{equation} 
See Figure \ref{fig4} for an illustration. 

Let $T_{x}$ be the tube around the vertical line passing through $x$ with $w(T_{x}) := \tfrac{1}{10} \beta r(x)$. Then
\begin{equation}\label{form18} T_{x} \, \setminus \, B(x,\tfrac{1}{2} \beta r(x)) \subset \mathcal{C}_{1}(x) \subset \mathcal{C}_{2\beta}(x) \subset \mathcal{C}_{\beta}(x). \end{equation}
(Recall that $2\beta \leq 1$.) In particular, \eqref{form18} implies $T_{x} \, \setminus \, B(x,r(x)) \subset \mathcal{C}_{2\beta}(x)$. Using this corollary, we observe that
\begin{align} B \cap T_{x} & \subset B(x,r(x)) \cup [(B \cap T_{x}) \, \setminus \, B(x,r(x))] \notag\\
& = B(x,r(x)) \cup [B \cap (T_{x} \, \setminus \, B(x,r(x)))] \notag\\
&\label{form17} \subset B(x,r(x)) \cup [B \cap \mathcal{C}_{2\beta}(x)] \stackrel{\eqref{form16}}{\subset} B(x,r(x)). \end{align} 

Choose a point $y(x) \in G \cap \mathcal{C}_{2\beta}(x)$ such that $|x-y(x)| \geq \tfrac{9}{10} r(x)$. A slightly more delicate geometric fact is that 
\begin{displaymath} T_{x} \subset \mathcal{C}_{\beta}(x) \cup \mathcal{C}_{\beta}(y(x)). \end{displaymath}
This is an exercise in elementary geometry, see Figure \ref{fig4} (or the proof in \cite[Lemma 15.14]{zbMATH01249699} for a more formal argument): the disc $B(x,\frac{1}{2} \beta r(x))$, and in particular the intersection $T_{x} \cap B(x,\frac{1}{2} \beta r(x))$, is contained in the cone $\mathcal{C}_{\beta}(y(x))$, whereas the rest of $T_{x}$ is contained in $\mathcal{C}_{\beta}(x)$, as already noted in \eqref{form18}. 
\begin{figure}[h!]
\begin{center}
\begin{overpic}{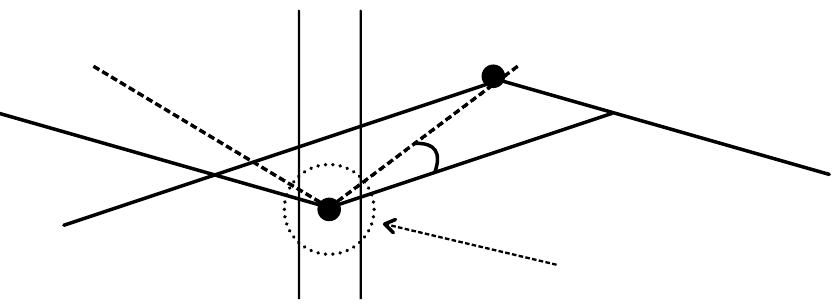}
\put(37,32){$T_{x}$}
\put(58,31){$y(x)$}
\put(68,3){$B(x,\frac{1}{2}\beta r(x))$}
\put(54,19){$\beta$}
\put(38,7){$x$}
\end{overpic}
\caption{Containing the tube $T_{x}$ in the union of the cones $\mathcal{C}_{\beta}(x)$ and $\mathcal{C}_{\beta}(y(x))$. The dotted cone illustrates $\mathcal{C}_{2\beta}(x) \ni y(x)$.}\label{fig4}
\end{center}
\end{figure}
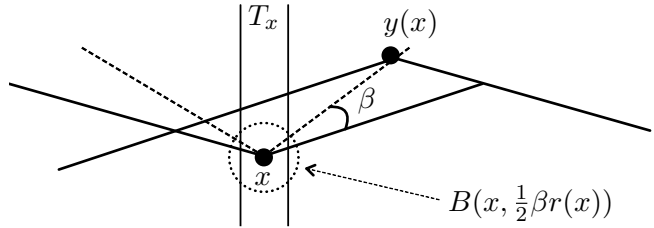
Consequently, using \eqref{form17}, the trivial inclusion $B(x,r(x)) \subset B(y(x),2r(x))$, and $x, y(x) \in G$, we have
\begin{align*} \mathcal{H}^{1}(B \cap T_{x}) & \leq \mathcal{H}^{1}(\mathcal{C}_{\beta}(y(x),2r(x)) \cap E) + \mathcal{H}^{1}(\mathcal{C}_{\beta}(x,r(x)) \cap E)\\
& \leq 2\varepsilon r(x) + \varepsilon r(x) \leq 30(\varepsilon/\beta) \cdot w(T_{x}). \end{align*}
We have now shown that every point $x \in B$ is contained on the central line of a vertical tube $T_{x}$ satisfying the estimate above. By the Besicovitch covering theorem, as in the proof of Lemma \ref{lemma1}, we may then find a countable, boundedly overlapping sub-family $\mathcal{T}$ of these tubes which still cover $B$. All the tubes intersect $B(1) \supset B$, so $\sum_{T \in \mathcal{T}} w(T) \lesssim 1$. It follows that 
\begin{displaymath} \mathcal{H}^{1}(B) \leq \sum_{T \in \mathcal{T}} \mathcal{H}^{1}(B \cap T) \leq \frac{30\varepsilon}{\beta} \sum_{T \in \mathcal{T}} w(T) \lesssim \frac{\varepsilon}{\beta}. \end{displaymath}
This completes the proof of Lemma \ref{lemma2}. \end{proof}

We are then ready to prove Proposition \ref{prop1}:

\begin{proof}[Proof of Proposition \ref{prop1}] Fix $\epsilon > 0$ as in the statement of the proposition, and set $\alpha'=\clip\alpha/2$. Define $\epsilon_{1} := \alpha \epsilon/C$ for a suitable absolute constant $C > 0$. By Lemma \ref{lemma1} applied to $\varepsilon=\epsilon_{1}$, we know that the set $R \subset E$ of bad points $x \in E$ with
\begin{displaymath} \maxconical{E_\alpha}{\alpha'}{x} \geq \epsilon_{1} \end{displaymath} 
satisfies 
\begin{displaymath} \mathcal{H}^{1}(R) \lesssim \delta \cdot \epsilon_{1}^{-1}\alpha^{-2} = C\delta \cdot \epsilon^{-1}\alpha^{-3}. \end{displaymath}
Since $\delta \le C_0\epsilon^{2}\alpha^{3}$, taking $C_0=C^{-2}$ gives $\mathcal{H}^{1}(R) \leq \epsilon/2$ (assuming that $C > 0$ was large enough). 

The set $G := E \, \setminus \, R$ satisfies the hypotheses of Lemma \ref{lemma2} (with $\beta=\alpha'=\clip\alpha/2$ and $\varepsilon=\epsilon_1$), so there exists a $\clip\alpha$-Lipschitz graph $\Gamma \subset \R$ over the $x$-axis such that $\mathcal{H}^{1}(G \, \setminus \, \Gamma) \lesssim \epsilon_{1}/\alpha = \epsilon/C$. If the constant $C > 0$ was chosen large enough, we see that
\begin{displaymath} \mathcal{H}^{1}(E \, \setminus \, \Gamma) \leq \mathcal{H}^{1}(R) + \mathcal{H}^{1}(G \, \setminus \, \Gamma) \leq \tfrac{\epsilon}{2} + \tfrac{\epsilon}{2} = \epsilon. \end{displaymath}
This concludes the proof of Proposition \ref{prop1}. \end{proof}

\section{Proof of Proposition \ref{prop2}}\label{s:prop2}
In this section we prove Proposition \ref{prop2}. Recall that we are assuming to be in "Case 2"; that is, $E$  cannot be exhausted, up to measure $\epsilon$, by a a constant number of consecutive sets $F_{k},F_{k + 1},\ldots,F_{k + \mathbf{C}_{\mathrm{sep}}}$ (recall this notation from Subsection \ref{ss:only-one-graph}).
More precisely, this meant that
\begin{equation}\label{form222} \mathcal{H}^{1}(E \, \setminus \, (F_{k} \cup \ldots \cup F_{k + \mathbf{C}_{\mathrm{sep}}})) \leq \epsilon. 
\end{equation}
 \textit{fails} for every $k$; thus we found an index pair $k,l \in \{0,\ldots,M_3\}$ with $|k - l| \geq \mathbf{C}_{\mathrm{sep}}$ such that 
\begin{equation}\label{form19b} 
        \mathcal{H}^{1}(F_{k})  \geq \alpha^{2\kappa} \quad \text{and} \quad \mathcal{H}^{1}(F_{l}) \geq \alpha^{2\kappa}. 
\end{equation}
Recall that all the minigraphs in $\mathcal{F}_{k}$ make an angle $\leq \alpha^{\kappa}$ with 
\begin{displaymath} L_{k} \coloneqq \spa(w_{k})= \spa(\cos(k \pi/M_{3}),\sin(k\pi/M_{3})) \end{displaymath}
and similarly all the minigraphs in $\mathcal{F}_{l}$ make an angle $\leq \alpha^{\kappa}$ with $L_{l} = \spa(w_{l})$.

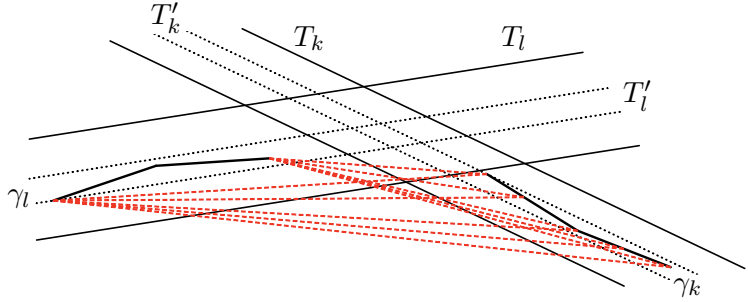
\begin{figure}[h!]
\begin{center}
\begin{overpic}[scale = 0.9]{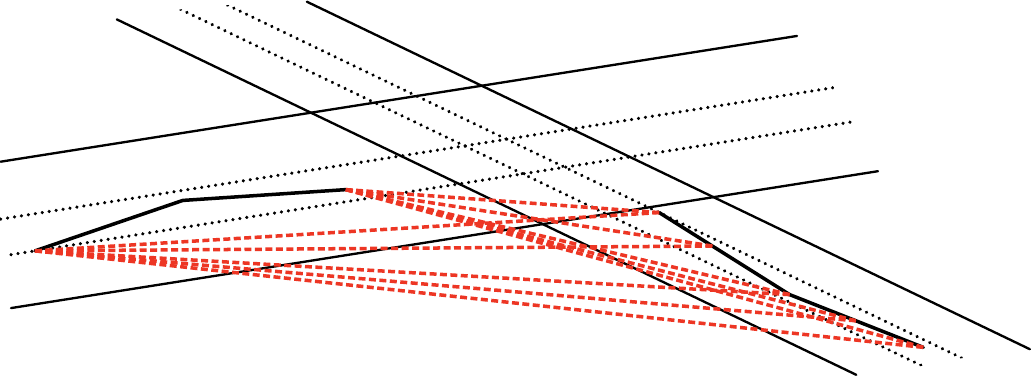}
\put(37,34){$T_{k}$}
\put(66,34){$T_{l}$}
\put(-3,12){$\gamma_{l}$}
\put(90,-0.5){$\gamma_{k}$}
\put(83,26){$T_{l}'$}
\put(17,37){$T_{k}'$}
\end{overpic}
\caption{A configuration where positively many lines hit $E$ twice.}\label{fig1}
\end{center}
\end{figure}
The existence of $F_{k}$ and $F_{l}$ will imply a configuration such as the one depicted in Figure \ref{fig1}. A more precise definition is given in the lemma below.
\begin{lemma}\label{lemma3} If \eqref{form19b} holds, then there exists an absolute constant $C \sim \mathbf{C}_{\mathrm{lip}}$ (the constant from Proposition \ref{prop1}) such that the following objects exist:
\begin{enumerate}
\item Affine lines $\ell_{k}$ and $\ell_{l}$ with $\angle(\ell_{k},L_{k}) \leq \alpha^{\kappa}$ and $\angle(\ell_{l},L_{l}) \leq \alpha^{\kappa}$.
\item Tubes $T_{k}' := \ell_{k}(C\alpha)$ and $T_{k} := \ell_{k}(\alpha^{1/2})$.
\item Tubes $T_{l}' := \ell_{l}(C\alpha)$ and $T_{l} := \ell_{l}(\alpha^{1/2})$.
\item $\mathbf{C}_{\mathrm{lip}}\alpha$-Lipschitz graphs $\gamma_{k},\gamma_{l}$ over the lines $\ell_{k},\ell_{l}$, respectively such that
\begin{displaymath} \gamma_{k} \cap B(1) \subset T_{k}' \quad \text{and} \quad \gamma_{l} \cap B(1) \subset T_{l}'. \end{displaymath}
\item Compact subsets 
\begin{equation}\label{form31} G_{k} \subset (E \cap \gamma_{k}) \, \setminus \, T_{l} \subset B(1) \quad \text{and} \quad G_{l} \subset (E \cap \gamma_{l}) \, \setminus \, T_{k} \subset B(1) \end{equation}
of measure $\mathcal{H}^{1}(G_{k}) \geq \alpha^{3}/C$ and $\mathcal{H}^{1}(G_{l}) \geq \alpha^{3}/C$.  
\end{enumerate}
\end{lemma}
Once the objects in Lemma \ref{lemma3} are found, it follows from a relatively simple geometric argument, presented below, that positively many lines intersect $E$ twice (the lines in question are depicted in red colour in Figure \ref{fig1}):
\begin{lemma}\label{l:lemma-cor}
There exists a set of lines $\mathcal{L}(G_{k},G_{l})$ of measure $\eta(\mathcal{L}(G_{k},G_{l})) \gtrsim \alpha^{7}$ such that $\ell \cap G_{k} \neq \emptyset$ and $\ell \cap G_{l} \neq \emptyset$ for all $\ell \in \mathcal{L}(G_{k},G_{l})$. In particular, since $G_{k},G_{l} \subset E$ are disjoint,
\begin{equation}\label{tform3} \int_{\mathcal{L}(E)} \#(E \cap \ell) - 1 \, d\eta(\ell) \gtrsim \eta(\mathcal{L}(G_{k},G_{l})) \gtrsim \alpha^{7}. \end{equation}
\end{lemma}

Proposition \ref{prop2} follows immediately by Lemma \ref{l:lemma-cor}. We will next derive Lemma \ref{l:lemma-cor} from Lemma \ref{lemma3}. (See Remark \ref{remark:eta-L-exact} and Appendix \ref{section:lines-spanned-by-rectifiable-curves} for an alternative proof of Lemma \ref{l:lemma-cor}.) 

\begin{proof}
The key geometric observation is the following: if $\ell \subset \R^{2}$ is any line with
\begin{displaymath} G_{k} \cap \ell \neq \emptyset \neq G_{l} \cap \ell, \end{displaymath}
then $\ell$ must make an angle $\gtrsim \alpha^{1/2}$ with both $\ell_{k}$ and $\ell_{l}$, see Figure \ref{fig1}: indeed, if for example $\angle(\ell,\ell_{l}) \ll \alpha^{1/2}$ and $\ell \cap G_{l} \neq \emptyset$, then $\ell \cap B(1) \subset T_{l}$, and hence $\ell \cap G_{k} = \emptyset$ by \eqref{form31}. It follows that both $\ell_{k},\ell_{l}$ are $C\alpha^{-1/2}$-graphs over $\ell^{\perp}$, for any line $\ell$ connecting $G_{k}$ and $G_{l}$. But since $\gamma_{k},\gamma_{l}$ were by definition $\mathbf{C}_{\mathrm{lip}}\alpha$-Lipschitz graphs over $\ell_{k},\ell_{l}$, it follows that also $\gamma_{k},\gamma_{l}$ are $C\alpha^{-1/2}$-Lipschitz graphs over $\ell^{\perp}$ (assuming that $\alpha > 0$ is small enough).

To prove the lower bound \eqref{tform3}, start by fixing $x \in G_{l} \subset \gamma_{l}$, recall that $\ell_{x,\theta} := \pi_{\theta}^{-1}\{\pi_{\theta}(x)\}$, and consider the set of directions
\begin{displaymath} \Theta(x,G_{k}) := \{\theta \in [0,\pi) : \ell_{x,\theta} \cap G_{k} \neq \emptyset\}. \end{displaymath}
With this notation, we claim that
\begin{equation}\label{form28} \mathcal{H}^{1}(\Theta(x,G_{k})) \gtrsim \alpha^{1/2}\mathcal{H}^{1}(G_{k}), \qquad x \in G_{l}. \end{equation}
Indeed, if $\{B(\theta_{j},r_{j})\}_{j \in \N}$ is an arbitrary cover of $\Theta(x,G_{k})$, then the tubes $\ell_{x,\theta_{j}}(Cr_{j})$ cover $G_{k}$, where $C > 0$ is an absolute constant. This is because $G_{k}$ is covered by the cones $C_{j} := \bigcup \{\ell_{x,\theta} : \theta \in B(\theta_{j},r_{j})\}$ by definition, and each intersection $G_{k} \cap C_{j} \subset B(1) \cap C_{j}$ is further covered by a tube of the form $\ell_{x,\theta_{j}}(Cr_{j})$. Now recall that $\gamma_{k} \supset G_{k}$ is an $\alpha^{-1/2}$-Lipschitz graph over each line $\ell_{x,\theta_{j}}^{\perp}$: this gives
\begin{displaymath} \alpha^{-1/2} \sum_{j \in \N} r_{j} \gtrsim \sum_{j \in \N} \mathcal{H}^{1}(G_{k} \cap \ell_{x,\theta_{j}}(r_{j})) \geq \mathcal{H}^{1}(G_{k}), \end{displaymath}
which implies \eqref{form28}.

We now infer from \eqref{form28} and Fubini's theorem that
\begin{align} \int_{0}^{\pi} & \mathcal{H}^{1}(\{x \in G_{l} : \theta \in \Theta(x,G_{k})\}) \, d\theta \notag\\
&\label{form29} = \int_{G_{l}} \mathcal{H}^{1}(\Theta(x,G_{k})) \, d\mathcal{H}^{1}(x) \gtrsim \alpha^{1/2}\mathcal{H}^{1}(G_{k})\mathcal{H}^{1}(G_{l}).  \end{align}
To proceed, write $G_{l}(\theta) := \{x \in G_{l} : \theta \in \Theta(x,G_{k})\}$. We claim that
\begin{equation}\label{form30} \mathcal{H}^{1}(G_{l}(\theta)) \neq 0 \quad \Longrightarrow \quad \mathcal{H}^{1}(\pi_{\theta}(G_{l}(\theta))) \gtrsim \alpha^{1/2}\mathcal{H}^{1}(G_{l}(\theta)), \qquad \theta \in [0,\pi). \end{equation}
This will complete the proof of the corollary, because \eqref{form29} then implies
\begin{displaymath} \int_{0}^{\pi} \mathcal{H}^{1}(\pi_{\theta}(G_{l}(\theta)) \, d\theta \stackrel{\eqref{form29}}{\gtrsim} \alpha \mathcal{H}^{1}(G_{k})\mathcal{H}^{1}(G_{l}) \stackrel{\textup{L. } \ref{lemma3}}{\gtrsim} \alpha^{7}, \end{displaymath}
and the left hand side above is a lower bound for $\eta(\mathcal{L}(G_{k},G_{l}))$. 

Finally, let us prove \eqref{form30}. If $\mathcal{H}^{1}(G_{l}(\theta)) \neq 0$, then $\theta \in \Theta(x,\gamma_{k})$ for at least one $x \in G_{l}$, which means that $\ell_{x,\theta} = \pi_{\theta}^{-1}\{\pi_{\theta}(x)\}$ intersects both $G_{k}$ and $G_{l}$. Thus, $\gamma_{l}$ is a $C\alpha^{-1/2}$-Lipschitz graph over the line $\ell_{x,\theta}^{\perp}$. Consequently, the relation $\mathcal{H}^{1}(\pi_{\theta}(H)) \gtrsim \alpha^{1/2}\mathcal{H}^{1}(H)$ holds for all $\mathcal{H}^{1}$ measurable subsets $H \subset \gamma_{l}$, in particular for $H := G_{l}(\theta)$.\end{proof}

\begin{remark}
\label{remark:eta-L-exact}
In fact, we have an exact expression for $\eta(\mathcal{L}(G_{k},G_{l}))$: 
\begin{align}
\label{eq:eta-L-exact}
\eta(\mathcal{L}(G_{k},G_{l}))
=
\iint_{G_k \times G_l} \frac{|\pi_{\theta(x_k, x_l)}(\tau_k(x_k))| \, |\pi_{\theta(x_k, x_l)}(\tau_l(x_l))|}{|x_k-x_l|} \, d(\mathcal{H}^1 \times \mathcal{H}^1)(x_k, x_l).
\end{align}
In \eqref{eq:eta-L-exact}, $\tau_k(x)$ denotes the unit tangent vector to $\gamma_k$ at $x \in \gamma_k$, and $\tau_l(x)$ is defined similarly. For distinct $x, x' \in \R^2$, $\theta(x, x')$ denotes the angle $\theta$ such that $\pi_\theta(x) = \pi_\theta(x')$.

Now we show how \eqref{eq:eta-L-exact} implies Lemma \ref{l:lemma-cor}. By the key geometric observation in the first paragraph of the proof of Lemma \ref{l:lemma-cor} and the fact that $G_k, G_l \subset B(1)$, the integrand in \eqref{eq:eta-L-exact} is $\gtrsim \frac{\alpha^{1/2} \alpha^{1/2}}{1} = \alpha$. Thus, $\eta(\mathcal{L}(G_{k},G_{l})) \gtrsim \alpha \mathcal{H}^1(G_k) \mathcal{H}^1(G_k) \gtrsim \alpha^7$.

We state and prove a more general form of \eqref{eq:eta-L-exact} in Appendix \ref{section:lines-spanned-by-rectifiable-curves}.
\end{remark}

The remainder of this section is devoted to constructing the objects listed in Lemma \ref{lemma3}. This is based on the assumption \eqref{form19}, that is, $\mathcal{H}^{1}(F_{k}) \geq \alpha^{2\kappa}$ and $\mathcal{H}^{1}(F_{l}) \geq \alpha^{2\kappa}$. Recall also that $F_{k},F_{l}$ were the unions of the minigraphs in $\mathcal{F}_{k}$ and $\mathcal{F}_{l}$. The minigraphs in $\mathcal{F}_{k}$ make an angle $\leq \alpha^{\kappa}$ with $L_{k}$, while the minigraphs in $\mathcal{F}_{l}$ make an angle $\leq \alpha^{\kappa}$ with $L_{l}$. Furthermore, $\angle(L_{k},L_{l}) \geq \mathbf{C}_{\mathrm{sep}}\alpha^{\kappa}$, so the minigraphs from $\mathcal{F}_{k}$ and $\mathcal{F}_{l}$ point in quantitatively different directions. We also recall that $\mathcal{F}_{k}$ (respectively $\mathcal{F}_{l})$ can be expressed as a union of certain consecutive families $\mathcal{E}_{i}$:
\begin{equation}\label{form20} \mathcal{F}_{k} = \mathcal{E}_{s} \cup \mathcal{E}_{s + 1} \cup \ldots \cup \mathcal{E}_{s + m} \quad \text{and} \quad \mathcal{F}_{l} = \mathcal{E}_{t} \cup \ldots \cup \mathcal{E}_{t + m}. \end{equation}
Some of these families may be empty, but not all, according to \eqref{form19b}. Of course
\begin{equation}\label{form21} m \lesssim \alpha^{-1}, \end{equation}
since there were no more than $\alpha^{-1}$ of the families $\mathcal{E}_{j}$ altogether. 

\subsection{Sketch of the proof} We now explain the proof strategy with a picture.
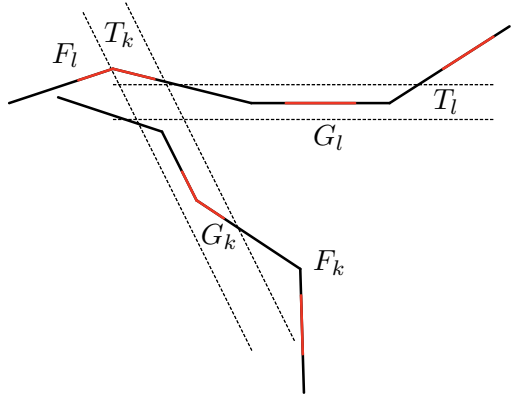
\begin{figure}[h!]
\begin{center}
\begin{overpic}[scale = 0.95]{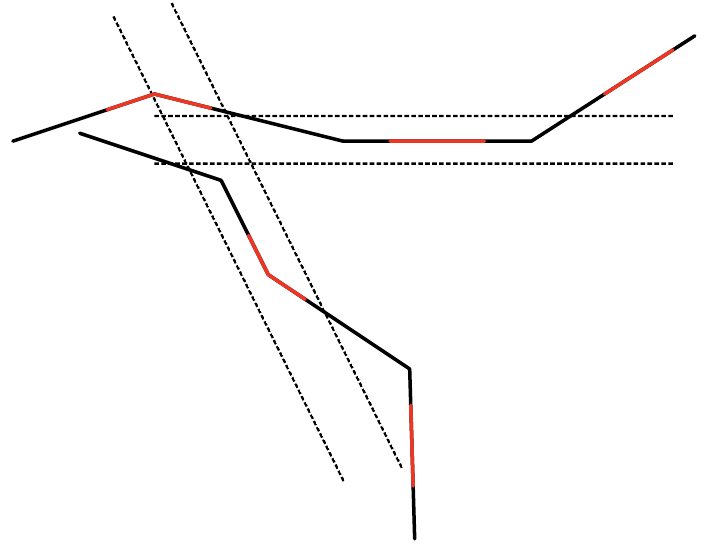}
\put(10,65){$F_{l}$}
\put(60,25){$F_{k}$}
\put(38.5,30){$G_{k}$}
\put(60,49){$G_{l}$}
\put(83,56){$T_{l}$}
\put(20,69){$T_{k}$}
\end{overpic}
\caption{Finding the graphs and tubes claimed by Lemma \ref{lemma3}.}\label{fig2}
\end{center}
\end{figure}
In Figure \ref{fig2}, we have depicted the sets $F_{k}$ and $F_{l}$, which are roughly speaking $\alpha^{\kappa}$-Lipschitz graphs over the lines $L_{k},L_{l}$ by Proposition \ref{prop1} (details will follow). Both $F_{k}$ and $F_{l}$ are, moreover, tiled by $\lesssim \alpha^{-1}$ of the sets $E_{j}$. Most of sets $E_{j}$ are (individually) contained on $\alpha$-Lipschitz graphs $\gamma_{j}$, by another application of Proposition \ref{prop1}. The red sets shown in Figure \ref{fig2} illustrate sets of the form
\begin{displaymath} G_{j} = E_{j} \cap \gamma_{j} \cap B_{j}, \end{displaymath}
where $B_{j}$ is some ball of radius $\alpha$ with the property that $\mathcal{H}^{1}(G_{j}) \sim_{\alpha} \mathcal{H}^{1}(E_{j})$. Each $G_{j}$ is contained in a tube $T_{j}$ of width $\alpha^{1/2}$ (or even a tube of width $\alpha$, which was also required in Lemma \ref{lemma3}). So, picking $G_{k} \subset F_{k}$ and $G_{l} \subset F_{l}$ arbitrarily, we would satisfy all the points (1)-(5) in Lemma \ref{lemma3}, except for the inclusions \eqref{form31}.

The problem is that if we pick $G_{k} \subset F_{k}$ and $G_{l} \subset F_{l}$ arbitrarily, the tube $T_{k}$ associated with $G_{k}$ might intersect $G_{l}$, or vice versa, violating \eqref{form31}. To satisfy \eqref{form31}, we need to pick $G_{k},G_{l}$ in such a way that the $G_{k}$-tube avoids $G_{l}$ and the $G_{l}$-tube avoids $G_{k}$. To achieve this, we roughly choose $3$ well-separated sets $G^{l}_{1},G^{l}_{2},G^{l}_{3} \subset F_{l}$, and $2$ further well-separated sets $G^{k}_{1},G^{k}_{2} \subset F_{k}$. 

Then, we use the "transversality" of the graphs $F_{k},F_{l}$ to deduce the following: each $G^{k}_{i}$-tube can intersect at most one of the sets $G^{l}_{j}$, and vice versa. At this point, we may deduce from the pigeonhole principle that there must exists a pair $(G^{k}_{i},G^{l}_{j})$ such that the $G^{k}_{i}$-tube does not intersect $G^{l}_{j}$, and the $G^{l}_{j}$-tube does not intersect $G^{k}_{i}$. Indeed, there are six pairs $(G^{k}_{i},G^{l}_{j})$, but only five tubes. This will complete the proof.\\

\subsection{Proof} We turn to the details. First, we apply Proposition \ref{prop1} to the sets $F_{k},F_{l}$, each of which can be written as a finite union of $\alpha^{\kappa}$-Lipschitz minigraphs  over the lines $L_{k},L_{l}$, respectively. It follows from the choice of constants $\delta = \epsilon^{70}/\mathbf{C}_{\mathrm{thm}}$ and $\alpha = (\epsilon/\mathbf{C}_{\mathrm{alp}})^{10}$ made in Section \ref{ss:only-one-graph} that $\delta \ll \alpha^{5\kappa}$, assuming that $\mathbf{C}_{\mathrm{thm}}$ is chosen sufficiently small compared to the absolute constant $\mathbf{C}_{\mathrm{alp}}$. Writing $\alpha^{5\kappa} = (\alpha^{\kappa})^{3}\alpha^{2\kappa}$, this means that the main hypothesis of Proposition \ref{prop1} is valid with constants "$\alpha^{\kappa}$" and "$\tfrac{1}{2}\alpha^{2\kappa}$" in place of "$\alpha$" and "$\epsilon$". It follows that there exist $\mathbf{C}_{\mathrm{lip}}\alpha^{\kappa}$-Lipschitz graphs $\Gamma_{k},\Gamma_{l}$ over $L_{k},L_{l}$, respectively, which cover most of $F_{k}$ and $F_{l}$ in the sense
\begin{displaymath} \mathcal{H}^{1}(F_{k} \, \setminus \, \Gamma_{k}) \leq \tfrac{1}{2}\alpha^{2\kappa} \stackrel{\eqref{form19}}{\leq} \tfrac{1}{2}\mathcal{H}^{1}(F_{k}) \quad \text{and} \quad \mathcal{H}^{1}(F_{l} \, \setminus \, \Gamma_{l}) \leq \tfrac{1}{2}\mathcal{H}^{1}(F_{l}). \end{displaymath}
We write $F_{k}' := F_{k} \cap \Gamma_{k}$ and $F_{l}' := F_{l} \cap \Gamma_{l}$. Next, recall from \eqref{form20} that 
\begin{displaymath} F_{k} = E_{s} \cup \ldots \cup E_{s + m} \quad \text{and} \quad F_{l} = E_{t} \cup \ldots \cup E_{t + m}, \end{displaymath}
and each $E_{j}$ is a finite union of $\alpha$-Lipschitz minigraphs $\mathcal{E}_{j}$ over a certain line (which makes an angle $\leq \alpha^{\kappa}$ with $L_{k}$). Applying Proposition \ref{prop1} again, for each $E_{j}$ with either $j \in \{s,\ldots,s + m\}$ or $j \in \{t,\ldots,t + m\}$, we find Lipschitz graphs $\gamma_{j}$ with constant $\leq \mathbf{C}_{\mathrm{lip}}\alpha$ and the property
\begin{displaymath} \mathcal{H}^{1}(E_{j} \, \setminus \, \gamma_{j}) \lesssim \alpha^{2}, \qquad s \leq j \leq s + m \text{ or } t \leq j \leq t + m. \end{displaymath}
For this application of Proposition \ref{prop1} to be legitimate, we need $\delta \ll \alpha^{3}(\alpha^{2})^{2} = \alpha^{7}$, which also follows from our choice of constants recalled above, taking $\mathbf{C}_{\mathrm{thm}} \gg \mathbf{C}_{\mathrm{alp}}^{70}$. We write $E_{j}' := E_{j} \cap \gamma_{j}$. With these choices, a major part of $F_{k}'$ is covered by the union of the graphs $\gamma_{j}$: indeed since $F_{k}' \subset F_{k} \subset \left(E_{s} \cup \ldots \cup E_{s + m}\right)$, we have
\begin{displaymath} \mathcal{H}^{1}\left(F_{k}' \, \setminus \, \bigcup_{j = 1}^{m} E_{s + j}' \right) \leq \sum_{j = 1}^{m} \mathcal{H}^{1}(E_{s + j} \, \setminus \, \gamma_{s + j}) \lesssim \sum_{j = 1}^{m} \alpha^{2} \stackrel{\eqref{form21}}{\lesssim} \alpha. \end{displaymath}
Since $\mathcal{H}^{1}(F_{k}') \gtrsim \mathcal{H}^{1}(F_{k}) \geq \alpha^{2\kappa}$, and $\kappa = \tfrac{1}{10}$, we infer that at least half of $F_{k}'$ is covered by the (subsets of) $\alpha$-Lipschitz graphs $E_{j}'$ with $s \leq j \leq s + m$. The same conclusion \emph{mutatis mutandis} holds for $F_{l}'$ and the sets $E_{j}'$ with $t \leq j \leq t + m$. We finally redefine
\begin{displaymath} F_{k} := F_{k}' \cap \bigcup_{j = 1}^{m} E_{s + j}' \quad \text{and} \quad F_{l} := F_{l}' \cap \bigcup_{j = 1}^{m} E_{t + j}'. \end{displaymath}
This should cause no confusion, since the original sets $F_{k},F_{l}$ will no longer be used. We list all the properties of $F_{k},F_{l}$ we will need in the sequel:
\begin{itemize}
\item $F_{k},F_{l} \subset E$ and $\mathcal{H}^{1}(F_{k}) \gtrsim \alpha^{2\kappa}$ and $\mathcal{H}^{1}(F_{l}) \gtrsim \alpha^{2\kappa}$ (compare with \eqref{form19}),
\item $F_{k}$ is covered by the Lipschitz graph $\Gamma_{k}$ over $L_{k}$ with constant $\leq \mathbf{C}_{\mathrm{lip}}\alpha^{\kappa}$,
\item $F_{l}$ is covered by the Lipschitz graph $\Gamma_{l}$ over $L_{l}$ with constant $\leq \mathbf{C}_{\mathrm{lip}}\alpha^{\kappa}$,
\item $F_{k}$ is covered by the union of $\lesssim \alpha^{-1}$ Lipschitz graphs $\gamma_{s},\ldots,\gamma_{s + m}$ with constant $\leq \mathbf{C}_{\mathrm{lip}}\alpha$ over certain lines $\ell_{s + j}$ making an angle $\leq \alpha^{\kappa}$ with $L_{k}$,
\item $F_{l}$ is covered by the union of $\lesssim \alpha^{-1}$ Lipschitz graphs $\gamma_{t},\ldots,\gamma_{t + m}$ with constant $\leq \mathbf{C}_{\mathrm{lip}}\alpha$ over certain lines $\ell_{t + j}$ making an angle $\leq \alpha^{\kappa}$ with $L_{l}$.
\end{itemize}

We have now defined carefully the objects $F_{k}$ and $F_{l}$ in Figure \ref{fig2}. In defining the objects $E_{k}$ and $E_{l}$ in the same picture, there is the technical problem that the "initial" sets $E_{j}$ need not be localised, as the picture suggests. This will be easily fixed by intersecting the initial sets $E_{j}$ with balls. First, using that $\mathcal{H}^{1}(F_{k}) \gtrsim \alpha^{2\kappa}$, we choose two special points $x_{1},x_{2} \in \bar{F}_{k}$ with the properties
\begin{equation}\label{form23} |x_{1} - x_{2}| \gtrsim \alpha^{2\kappa} \quad \text{and} \quad \mathcal{H}^{1}(F_{k} \cap B(x_{j},\alpha)) \geq \alpha^{2} \text{ for } j \in \{1,2\}. \end{equation}
This can be arranged, because the set of points $x \in F_{k}$ with $\mathcal{H}^{1}(F_{k} \cap B(x,\alpha)) \leq \alpha^{2}$ has total length at most $\lesssim \alpha \ll \mathcal{H}^{1}(F_{k})$. Thus, the admissible points for the second condition in \eqref{form23} have total length $\geq \tfrac{1}{2}\mathcal{H}^{1}(F_{k}) \gtrsim \alpha^{2\kappa}$. Then, to finish the selection, it remains to pick two of these points with separation $\alpha^{2\kappa}$: this is possible because $F_{k}$ lies on a Lipschitz graph with constant $\leq 1$, so in particular $\mathcal{H}^{1}(F_{k} \cap B(x,r)) \lesssim r$ for all $r > 0$.

Next, we move attention from $F_{k}$ to $F_{l}$. This time we pick $3$ special points $y_{1},y_{2},y_{3} \in F_{l}$ with properties similar to those in \eqref{form23}:
\begin{equation}\label{form24} |y_{i} - y_{j}| \gtrsim \alpha^{2\kappa} \text{ for } i \neq j \quad \text{and} \quad \mathcal{H}^{1}(F_{l} \cap B(y_{j},\alpha)) \geq \alpha^{2} \text{ for } j \in \{1,2,3\}. \end{equation}
The details of the selection are the same as we have seen above.

Next, recall that both $F_{k}$ and $F_{l}$ can be written as a finite union of (subsets of) $\mathbf{C}_{\mathrm{lip}}\alpha$-Lipschitz graphs: the covering graphs for $F_{k}$ were denoted $\gamma_{s},\ldots,\gamma_{s + m}$ and the covering graphs for $F_{l}$ were denoted $\gamma_{t},\ldots,\gamma_{t + m}$, where $m \lesssim \alpha^{-1}$. Since $\mathcal{H}^{1}(F_{k} \cap B(x_{1},\alpha)) \geq \alpha^{2}$, at least one of the graphs $\gamma_{s},\ldots,\gamma_{s + m}$ must have large intersection with $F_{k} \cap B(x_{1},\alpha)$. We denote this graph by $\gamma^{k}_{1}$; then we have
\begin{equation}\label{form26} \mathcal{H}^{1}(F_{k} \cap \gamma^{k}_{1} \cap B(x_{1},\alpha)) \gtrsim \alpha^{3}. \end{equation}
We find similarly a graph $\gamma^{k}_{2} \in \{\gamma_{s},\ldots,\gamma_{s + m}\}$ such that $\mathcal{H}^{1}(F_{k} \cap \gamma^{k}_{2} \cap B(x_{2},\alpha)) \gtrsim \alpha^{3}$. Then, we also repeat the argument for the three balls $B(y_{j},\alpha)$: we find three graphs $\gamma^{l}_{1},\gamma^{l}_{2},\gamma^{l}_{3} \in \{\gamma_{t},\ldots,\gamma_{t + m}\}$ with the property
\begin{equation}\label{form27} \mathcal{H}^{1}(F_{l} \cap B(y_{j},\alpha) \cap \gamma^{l}_{j}) \gtrsim \alpha^{3}, \qquad 1 \leq j \leq 3. \end{equation}
The sets 
\begin{align}\label{form34} & G_{i}^{k} := F_{k} \cap \gamma_{i}^{k} \cap B(x_{i},\alpha),\, \, \, i=1,2, \quad \text{and} \nonumber \\
& G_{j}^{l} := F_{l} \cap \gamma_{j}^{l} \cap B(y_{j},\alpha), \,\,\, j=1,2,3 \end{align}
are the ones we informally discussed below Figure \ref{fig2}.

Next, we associate the lines and tubes (required by Lemma \ref{lemma3}) to the sets $G_{i}^{k},G_{j}^{l}$. We associate to each graph $\gamma_{i}^{k}$ or $\gamma_{j}^{l}$ an affine line $\ell_{i}^{k}$ or $\ell_{j}^{l}$ with the following properties:
\begin{itemize}
\item $\gamma_{i}^{k}$ is a $\mathbf{C}_{\mathrm{lip}}\alpha$-Lipschitz graph over $\ell_{i}^{k}$ for $i \in \{1,2\}$,
\item $\gamma_{j}^{l}$ is a $\mathbf{C}_{\mathrm{lip}}\alpha$-Lipschitz graph over $\ell_{j}^{l}$ for $j \in \{1,2,3\}$,
\item The lines are chosen so that
\begin{displaymath} G_{j}^{k} \subset \ell^{k}_{i}(C\alpha) \text{ for } i \in \{1,2\} \quad \text{and} \quad G_{j}^{l} \subset \ell_{j}^{l}(C\alpha)  \text{ for } j \in \{1,2,3\}, \end{displaymath} 
\end{itemize}
where $C \sim \mathbf{C}_{\mathrm{lip}}$. We now define
\begin{displaymath} (T^{k}_{i})' := \ell_{i}^{k}(C\alpha) \quad \text{and} \quad T^{k}_{i} := \ell_{i}^{k}(\alpha^{1/2}) \end{displaymath}
for $i \in \{1,2\}$, and similarly
\begin{displaymath} (T^{l}_{j})' := \ell_{j}^{l}(C\alpha) \quad \text{and} \quad T^{l}_{j} := \ell_{j}^{l}(\alpha^{1/2}) \end{displaymath}
for $j \in \{1,2,3\}$. Thus, $G^{k}_{i} \subset (T^{k}_{i})' \subset T^{k}_{i}$ and $G^{l}_{j} \subset (T^{l}_{j})' \subset T^{l}_{j}$. Since moreover $\mathcal{H}^{1}(G^{k}_{i}) \gtrsim \alpha^{3}$ and $\mathcal{H}^{1}(G^{l}_{j}) \gtrsim \alpha^{3}$ by \eqref{form26}-\eqref{form27}, \textbf{any} pair $(G^{k}_{i},G^{l}_{j})$ (with associated lines and tubes) would now satisfy all the requirements of Lemma \ref{lemma3}, except perhaps the inclusions \eqref{form31}.

We will now use the pigeonhole principle to show that at least one of the pairs $(G^{k}_{i},G^{l}_{j})$ also satisfies the inclusions \eqref{form31}. The main geometric observation is the following:
\begin{equation}\label{form25} \diam(T^{k}_{i} \cap \Gamma_{l}) \lesssim \alpha^{1/2 - \kappa} \quad \text{and} \quad \diam(T^{l}_{j} \cap \Gamma_{k}) \lesssim \alpha^{1/2 - \kappa}. \end{equation}
\begin{figure}[h!]
\begin{center}
\begin{overpic}[scale = 0.95]{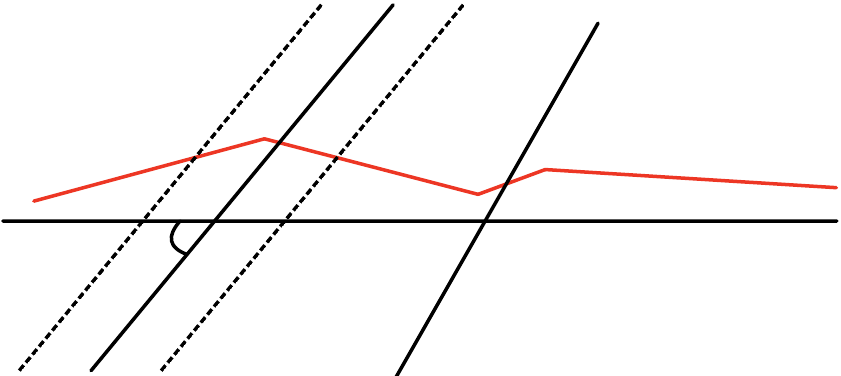}
\put(75,26){$\Gamma_{l}$}
\put(75,13){$L_{l}$}
\put(60,35){$L_{k}$}
\put(31,10){$T^{k}_{i}$}
\put(13,10){$\ell^{k}_{j}$}
\end{overpic}
\caption{Transversality of $T_{i}^{k}$ and $\Gamma_{l}$. The angle between $\ell_{j}^{k}$ and $L_{l}$ is $\gtrsim C\alpha^{\kappa}$.}\label{fig3}
\end{center}
\end{figure}
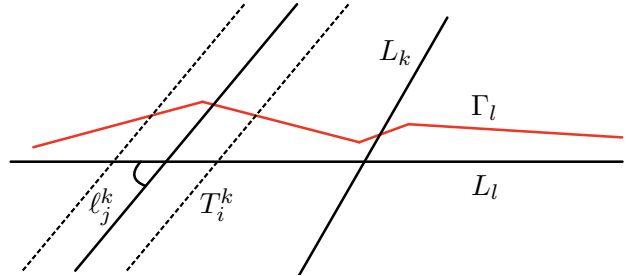
The first inequality holds for $i \in \{1,2\}$, the second for $j \in \{1,2,3\}$. The proof of \eqref{form25} is contained in Figure \ref{fig3}. Recall that $T_{i}^{k}$ is an $\alpha^{1/2}$-tube around a certain line $\ell^{k}_{i}$ with $\angle(\ell^{k}_{i},L_{k}) \leq \alpha^{\kappa}$. On the other hand, $\angle(L_{k},L_{l}) \geq \mathbf{C}_{\mathrm{sep}}\alpha^{\kappa}$, so also $\angle(\ell_{i}^{k},L_{l}) \geq (\mathbf{C}_{\mathrm{sep}} - 1)\alpha^{\kappa}$. Finally, $\Gamma_{l}$ is a $\mathbf{C}_{\mathrm{lip}}\alpha^{\kappa}$-Lipschitz graph over $L_{l}$, so every tangent of $\Gamma_{l}$ makes an angle $\gtrsim \mathbf{C}_{\mathrm{sep}}\alpha^{\kappa}$ with $\ell^{k}_{i}$, since we chose $\mathbf{C}_{\mathrm{sep}}$ much larger than $\mathbf{C}_{\mathrm{lip}}$ in Section \ref{ss:only-one-graph}. Thus $\Gamma_{l}$ is an $\alpha^{-\kappa}$-Lipschitz graph over $(\ell_{j}^{k})^{\perp}$. It follows that
\begin{displaymath} \diam(T_{i}^{k} \cap \Gamma_{l}) \leq \mathcal{H}^{1}(T_{i}^{k} \cap \Gamma_{l}) \lesssim \alpha^{1/2 - \kappa}. \end{displaymath}
Now that we have proved \eqref{form25}, recall from \eqref{form24} the three balls $B(y_{j},\alpha)$, all of which were centred at $y_{j} \in F_{l} \subset \Gamma_{l}$, and whose centres $y_{j}$ had pairwise separation $\gtrsim \alpha^{2\kappa}$. Since $\kappa = \tfrac{1}{10}$, we have $\alpha^{1/2 - \kappa} \ll \alpha^{2\kappa}$ for $\alpha > 0$ small enough (or in other words assuming that the constant $\mathbf{C}_{\mathrm{alp}} > 0$ is chosen large enough), and therefore \eqref{form25} implies that 
\begin{equation}\label{form32} \# \{j \in \{1,2,3\} : T^{k}_{i} \cap B(y_{j},\alpha) \neq \emptyset\} \leq 1, \qquad i \in \{1,2\}. \end{equation}
By a similar argument,
\begin{equation}\label{form33} \# \{i \in \{1,2\} : T^{l}_{j} \cap B(x_{i},\alpha) \neq \emptyset\} \leq 1, \qquad j \in \{1,2,3\}. \end{equation}
We finally claim, as a consequence of \eqref{form32}-\eqref{form33} and the pigeonhole principle, that there exists a pair of balls $(B(x_{i_0},\alpha),B(y_{j_0},\alpha))$, for some $i_0 \in \{1,2\}$ and $j_0 \in \{1,2,3\}$ with the property
\begin{equation}\label{form35} T_{i_0}^{k} \cap B(y_{j_0},\alpha) = \emptyset \quad \text{and} \quad  T_{j_0}^{l} \cap B(x_{i_0},\alpha) = \emptyset. \end{equation}
This, by definition, yields
\begin{displaymath} G_{i_0}^{k} \stackrel{\eqref{form34}}{\subset} B(x_{i_0},\alpha) \, \setminus \, T_{j_0}^{l} \quad \text{and} \quad G_{j_0}^{l} \stackrel{\eqref{form34}}{\subset} B(y_{j_0},\alpha) \, \setminus \, T_{i_0}^{k}, \end{displaymath}
which (combined with \eqref{form34}) completes the proof of the inclusions \eqref{form31}, and Lemma \ref{lemma3}.

To prove \eqref{form35}, consider the bi-partite graph with $5$ vertices $\{v_{1},v_{2}\} \cup \{w_{1},w_{2},w_{3}\}$ and the following edge set. 
\begin{itemize}
\item For $i \in \{1,2\}$ and $j \in \{1,2,3\}$, the edge $(v_{i},w_{j})$ is included if $T_{i}^{k} \cap B(y_{j},\alpha) \neq \emptyset$. 
\item For $j \in \{1,2,3\}$ and $i \in \{1,2\}$, the edge $(w_{j},v_{i})$ is included if $T_{j}^{l} \cap B(x_{i},\alpha) \neq \emptyset$.
\end{itemize}
Now, \eqref{form32}-\eqref{form33} can be restated as follows: for $v_{i}$ fixed, there can be at most one edge $(v_{i},w_{j})$, and for $w_{i}$ fixed, there can be at most one edge $(w_{i},v_{j})$. Thus, the edge set contains at most $5$ edges. On the other hand, the product set $\{v_{1},v_{2}\} \times \{w_{1},w_{2},w_{3}\}$ contains $6$ elements, so there must be a pair $\{v_{i},w_{j}\}$ so that neither $(v_{i},w_{j})$ nor $(w_{j},v_{i})$ lies in the edge set. This is equivalent to \eqref{form35}. This completes the proof of Lemma \ref{lemma3}.

\section{The grid example}\label{sec:grid}
In this section we provide an example showing that Theorem \ref{main} is optimal in the sense that the assumption $\Fav(E) \geq \Fav(L) - \delta$ cannot be relaxed to $\Fav(E) \geq \delta$.
\begin{proposition}\label{gridProp}
    There exists an absolute constant $\delta > 0$ and a sequence of compact rectifiable sets $E_n\subset [0,1]^2\subset\mathbb{R}^2$ such that:
\begin{itemize}
    \item[\textup{(1)}] $\mathcal{H}^1(E_n)=1,$
    \item[\textup{(2)}] $\Fav(E_n)\ge \delta$,
    \item[\textup{(3)}] for any $\alpha\in [2n^{-2},1)$ and any curve $\Gamma$ with $\mathcal{H}^1(\Gamma\cap E_n)\ge\alpha$ we have $\mathcal{H}^1(\Gamma)\gtrsim \alpha n$.
\end{itemize}
In particular, property \textup{(3)} implies that if $M \geq 1$, then for any $M$-Lipschitz graph $\Gamma$ we have $\mathcal{H}^1(\Gamma\cap E_n)\lesssim Mn^{-1}.$
\end{proposition}

We begin the construction. Fix an integer $n\ge 2$, and let $[n]\coloneqq\{1,\dots,n\}$. For any $j=(k,l)\in [n]^2$ set
\begin{equation}\label{eq:grid}
    x_j= \bigg(\frac{k}{n+1},\, \frac{l}{n+1}\bigg)
\end{equation}
and
\begin{equation*}
    B_j = B\bigg(x_j,\, \frac{1}{2\pi n^2}\bigg).
\end{equation*}
Note that $B_j\subset [0,1]^2$ and if $i,j\in  [n]^2,\ i\neq j$, then
\begin{equation}\label{eq:separation}
    \dist(B_i,\, B_j)\ge \frac{1}{n+1} - \frac{2}{2\pi n^2} \ge \frac{1}{2n}.
\end{equation}
Define $S_j=\partial B_j$, and observe that $\mathcal{H}^1(S_j)=n^{-2}$. 

We define the set $E_n$ as
\begin{equation*}
    E_n \coloneqq \bigcup_{j\in [n]^2} S_j.
\end{equation*}
Since $\mathcal{H}^1(S_j)=n^{-2}$, we have $\mathcal{H}^1(E_n)=1.$ This verifies property (1) for $E_n$. It is also clear that $E_n$ is compact and rectifiable.

Now we check property (3). We will use the following result.
\begin{lemma}[Lemma 3.7 from \cite{schul2007subsets}]\label{lem:Schul}
    Any compact connected set $\Gamma\subset\R^2$ with $\mathcal{H}^1(\Gamma)<\infty$ can be parametrized with $\gamma:[0,1]\to\R^2$ such that $\gamma([0,1])=\Gamma$ and $\lip(\gamma)\le 32\,\mathcal{H}^1(\Gamma)$. 
\end{lemma}
\begin{lemma}
    For any $\alpha\in [2n^{-2},1)$ and any curve $\Gamma$ with $\mathcal{H}^1(\Gamma\cap E_n)\ge\alpha$ we have $\mathcal{H}^1(\Gamma)\gtrsim \alpha n$.
\end{lemma}
\begin{proof}
    Suppose that $\alpha\in [2n^{-2},1)$ and let $\Gamma$ be a curve with $\mathcal{H}^1(\Gamma\cap E_n)\ge\alpha$. 
    Since each circle $S_j$ comprising $E_n$ has length $n^{-2}$, we get that $\Gamma$ intersects at least $\alpha n^2$ different circles. Let $J_0\subset[n]^2$ be the set of indices such that for $j\in J_0$ we have $\Gamma\cap S_j\neq\varnothing$, so that
    \begin{equation}\label{eq:J0}
        N\coloneqq \#J_0\ge \alpha n^2.
    \end{equation}
    
    To estimate $\mathcal{H}^1(\Gamma)$, we are going to use \eqref{eq:J0} together with the fact that the circles $S_j$ are centered on a well-separated grid \eqref{eq:grid}, \eqref{eq:separation}. We provide the details below.
    
    Let $\gamma$ be the parametrisation of the curve $\Gamma$ given by Lemma \ref{lem:Schul}. Without loss of generality, we may assume that the curve $\Gamma$ begins and ends on $E_n$, i.e., $\gamma(0),\gamma(1)\in \Gamma\cap E_n$. For all $j\in J_0$ we choose a point $y_j\in \Gamma \cap S_j$, and let $t_j\in [0,1]$ be such that $\gamma(t_j)=y_j$ ($\gamma$ might be non-injective, in which case $t_j$ is non-unique, but in this case we pick $t_{j}$ arbitrarily among the admissible options). The only constraint we make on our choice of $\{y_j\}_{j\in J_0}$ is so that $\gamma(0), \gamma(1)\in \{y_j\}_{j\in J_0}$. For convenience, we relabel the points $t_j$ in ``ascending order'': for all $i\in \{1,\dots,N\}$ we set $t_i:=t_j$ for some $j\in J_0$, in such a way that $t_1< t_2<\dots<t_{N}$. We relabel in a similar way $y_j$ and $S_j$. 
    
    Recalling that the circles $S_j$ are centered on a grid \eqref{eq:grid}, it follows from the separation property \eqref{eq:separation} that for any $i\in \{1,\dots,N\}$
    \begin{equation*}
        \frac{1}{2n}\le |y_{i+1}-y_{i}|= |\gamma(t_{i+1})-\gamma(t_{i})|\le \lip(\gamma)\cdot|t_{i+1}-t_{i}| = \lip(\gamma)\cdot(t_{i+1}-t_{i}).
    \end{equation*}
    Summing over $i\in \{1,\dots,N-1\}$ we get 
    \begin{equation*}
        \frac{N-1}{2n}\le \lip(\gamma)\cdot(t_{N}-t_{1})\le 32\,\mathcal{H}^1(\Gamma)\cdot(t_{N}-t_{1}).
    \end{equation*}
    Since we we assumed $\gamma(0), \gamma(1)\in \{y_j\}_{j\in J_0}$, we get that $t_{N}=1$ and $t_{1}=0.$ Thus,
    \begin{equation*}
        32\,\mathcal{H}^1(\Gamma) \ge \frac{N-1}{2n} \overset{\eqref{eq:J0}}{\ge} \frac{\alpha n^2-1}{2n}\ge \frac{\alpha n}{4}.
    \end{equation*}
 This completes the proof of the lemma. \end{proof}

It remains to prove the property (2), that is, $\Fav(E_n)\ge \delta$. Let 
\begin{equation*}
    G_n = \bigcup_{j\in [n]^2} B_j,
\end{equation*}
so that $E_n=\partial G_n$. Note that $\Fav(E_n)=\Fav(G_n)$. We define an auxiliary measure
\begin{equation*}
    \mu = \mu_n=\frac{1}{\mathcal{L}^2(G_n)}\mathcal{L}^2|_{G_n}.
\end{equation*}
Recall that the $1$-energy of $\mu$ is defined as
\begin{equation*}
    I_1(\mu) = \iint \frac{1}{|x-y|}\, d\mu(x)d\mu(y).
\end{equation*}
\begin{lemma}
    We have 
    \begin{equation*}
        I_1(\mu) \lesssim 1.
    \end{equation*}
    As a consequence, 
    \begin{equation}\label{eq:Favbound}
        \Fav(E_n)=\Fav(G_n)\gtrsim 1.
    \end{equation}
\end{lemma}
\begin{proof}
    We write
    \begin{multline*}
        I_1(\mu) = \iint \frac{1}{|x-y|}\, d\mu(x)d\mu(y) = \sum_{i,j\in[n]^2} \int_{B_i}\int_{B_j}\frac{1}{|x-y|}\, d\mu(x)d\mu(y)\\
        \sum_{i\in[n]^2} \int_{B_i}\int_{B_i}\frac{1}{|x-y|}\, d\mu(x)d\mu(y) + \sum_{i,j\in[n]^2,\, i\neq j} \int_{B_i}\int_{B_j}\frac{1}{|x-y|}\, d\mu(x)d\mu(y) = A_1 + A_2.
    \end{multline*}
    To estimate $A_1$ we note that for any $i\in[n]^2$ and any fixed $x\in B_i$
    \begin{multline*}
        \int_{B_i}\frac{1}{|x-y|}\, d\mu(y) \le \sum_{k=\lfloor\log_2 n^2\rfloor}^\infty \int_{B(x,2^{-k})\setminus B(x,2^{-k-1})}\frac{1}{|x-y|}\, d\mu(y)\\
        \sim \sum_{k=\lfloor\log_2 n^2\rfloor}^\infty 2^k \mu(B(x,2^{-k})\setminus B(x,2^{-k-1}))
        \lesssim \frac{1}{\mathcal{L}^2(G_n)} \sum_{k=\lfloor\log_2 n^2\rfloor}^\infty 2^k \mathcal{L}^2(B(x,2^{-k}))\\
        \sim n^2 \sum_{k=\lfloor\log_2 n^2\rfloor}^\infty 2^k\cdot 2^{-2k}\sim 1.
    \end{multline*}
    Hence, 
    \begin{equation*}
        A_1 = \sum_{i\in[n]^2} \int_{B_i}\int_{B_i}\frac{1}{|x-y|}\, d\mu(x)d\mu(y) \lesssim \sum_{i\in[n]^2} \mu(B_i) = 1.
    \end{equation*}
    
    We move on to estimating $A_2$. Let $Q_j$ denote the square centered at $x_j$ with sidelength $1/(n+1)$. Note that $B_j\subset Q_j$, and the squares $Q_j,\ j\in [n]^2$ are pairwise disjoint. If $x\in B_i$ and $y\in B_j$, with $i\neq j$, then $|x-y|\sim \dist(B_i, B_j)\sim |x-z|$ for any $z\in Q_j$. It follows that for a fixed $x\in B_i$
    \begin{equation*}
        \int_{B_j}\frac{1}{|x-y|}\, d\mu(y) \sim \dist(B_i, B_j)^{-1}\,\mu(B_j)\sim \dist(B_i, B_j)^{-1}\,\mathcal{L}^2(Q_j)\sim \int_{Q_j}\frac{1}{|x-z|}\, d\mathcal{L}^2(z)
    \end{equation*}
    Summing over $j\in[n]^2\setminus \{i\}$ yields
    \begin{multline*}
        \sum_{j\in[n]^2\setminus \{i\}} \int_{B_j}\frac{1}{|x-y|}\, d\mu(y) \sim \sum_{j\in[n]^2\setminus \{i\}} \int_{Q_j}\frac{1}{|x-z|}\, d\mathcal{L}^2(z)\le \int_{[-1,2]^2}\frac{1}{|x-z|}\, d\mathcal{L}^2(z)\\
        \lesssim\sum_{k=-1}^{\infty} \int_{B(x,2^{-k})\setminus B(x,2^{-k-1})} 2^k\, d\mathcal{L}^2(z) \lesssim 1.
    \end{multline*}
    Thus,
    \begin{equation*}
        A_2 =\sum_{i\in[n]^2} \int_{B_i}\bigg(\sum_{j\in[n]^2\setminus\{i\}}\int_{B_j}\frac{1}{|x-y|}\, d\mu(y)\bigg)d\mu(x)\lesssim \sum_{i\in[n]^2}\mu(B_i)=1.
    \end{equation*}
    It follows that $I_1(\mu)\lesssim 1.$ 
    
    To see \eqref{eq:Favbound}, we use Theorem 4.3 from \cite{mattila2015fourier} to conclude that
    \begin{equation*}
        \Fav(E_n)=\Fav(G_n)\gtrsim \frac{1}{I_1(\mu)}\gtrsim 1. \end{equation*}
    This concludes the proof of Proposition \ref{gridProp}. 
\end{proof}

\appendix

\section{Lines spanned by rectifiable curves}
\label{section:lines-spanned-by-rectifiable-curves}

We state and prove a generalization of \eqref{eq:eta-L-exact}, which was mentioned in Remark \ref{remark:eta-L-exact}.

\begin{lemma}
\label{lemma:lines-spanned-by-rectifiable-curves}
Let $\gamma_1, \gamma_2 \subset \R^2$ be rectifiable curves. For $\mathcal{H}^1$ almost every $x \in \gamma_i$, let $\tau_i(x)$ denote the unit tangent vector to $\gamma_i$ at $x$. (The choice of direction is irrelevant.) Then for any $G_1 \subset \gamma_1$ and $G_2 \subset \gamma_2$, we have
\begin{align*}
&\int_{\mathcal{A}}
\# \{(x_1, x_2) \in G_1 \times G_2 : x_1 \neq x_2 \text{ and } x_1, x_2 \in \ell\} \, d\eta(\ell)
\\
&\qquad=
\iint_{G_1 \times G_2} \frac{|\pi_{\theta(x_1, x_2)}(\tau_1(x_1))| \, |\pi_{\theta(x_1, x_2)}(\tau_2(x_2))|}{|x_1-x_2|} \, d(\mathcal{H}^1 \times \mathcal{H}^1)(x_1, x_2)
\end{align*}
where $\theta(x_1, x_2)$ denotes the angle $\theta$ such that $\pi_\theta(x_1) = \pi_\theta(x_2)$.
\end{lemma}

\begin{proof}
Let $\phi_i(s)$ be a parametrization of $\gamma_i$ by arclength. Consider the map $\Psi : (s_1, s_2) \mapsto (\theta, t)$ defined implicitly by 
\begin{align}
\label{eq:def-Psi}    
\pi_\theta(\phi_1(s_1)) = \pi_\theta(\phi_2(s_2)) = t.
\end{align}
By the change of variables formula,
\begin{align*}
&
\int_{\mathcal{A}}
\# \{(x_1, x_2) \in G_1 \times G_2 : x_1 \neq x_2 \text{ and } x_1, x_2 \in \ell\} \, d\eta(\ell)
\\
&=
\int_{[0,\pi] \times \mathbb{R}}
\# \{(x_1, x_2) \in G_1 \times G_2 : x_1 \neq x_2 \text{ and } x_1, x_2 \in \pi_\theta^{-1}(t)\} \, d\mathcal{H}^2(\theta,t)
\\
&=\iint_{s_1 \in \phi_1^{-1}(G_1), s_2 \in \phi_2^{-1}(G_2)}
J\Psi(s_1, s_2)
\, ds_1 \, ds_2,
\end{align*}
where $J\Psi$ denotes the Jacobian determinant of $\Psi$. (Note that the set $\{(s_1, s_2) : \phi_1(s_1) = \phi_2(s_2)\}$ has $\mathcal{H}^2$-measure zero.) 

We now prove that
\begin{align}
\label{eq:jacobian}
J\Psi(s_1, s_2)
:=
\operatorname{abs}
\begin{vmatrix}
\partial_{s_1} \theta & \partial_{s_2} \theta 
\\
\partial_{s_1} t & \partial_{s_2} t 
\end{vmatrix}
=
\frac{|\pi_{\theta(s_1,s_2)}(\gamma_1'(s_1))|\ |\pi_{\theta(s_1, s_2)}(\gamma_2'(s_2))|}{|\gamma_1(s_1) - \gamma_2(s_2)|}.
\end{align}
Note that this would finish the proof of the lemma. To show \eqref{eq:jacobian}, define $e_\theta = (\cos\theta, \sin\theta)$ and $e_\theta^\perp = \frac{d}{d\theta} e_\theta = (-\sin\theta, \cos\theta)$. By differentiating \eqref{eq:def-Psi} with respect to $s_1$ and $s_2$, we obtain
\begin{alignat*}{2}
e_\theta \cdot \phi_1'(s_1) + e_\theta^\perp \cdot \phi_1(s_1) \partial_{s_1}\theta 
&= 
e_\theta^\perp \cdot \phi_2(s_2) \partial_{s_1}\theta 
&&= 
\partial_{s_1}t
\\
e_\theta \cdot \phi_2'(s_2) + e_\theta^\perp \cdot \phi_2(s_2) \partial_{s_2}\theta 
&= 
e_\theta^\perp \cdot \phi_1(s_1) \partial_{s_2}\theta 
&&= 
\partial_{s_2}t.
\end{alignat*}
The two equalities on the left give
\begin{align*}
|\partial_{s_i}\theta|
=
\frac{|e_\theta \cdot \phi_i'(s_i)|}{|e_\theta^\perp \cdot (\phi_1(s_1) - \phi_2(s_2))|}
\text{ for } i = 1,2
\end{align*}
which, when combined with the two equalities on the right, give
\begin{align*}
J\Psi(s_1, s_2)
=
|\partial_{s_1} \theta| \, |\partial_{s_2} \theta| \, |e_\theta^\perp \cdot (\phi_1(s_1) - \phi_2(s_2))|
=
\frac{|e_\theta \cdot \phi_1'(s_1)|\, |e_\theta \cdot \phi_2'(s_2)|}{|e_\theta^\perp \cdot (\phi_1(s_1) - \phi_2(s_2))|}.
\end{align*}
Finally, observe that $e_\theta \cdot (\phi_1(s_1) - \phi_2(s_2)) = 0$ by the definition of $\Psi$, which implies $|e_\theta^\perp \cdot (\phi_1(s_1) - \phi_2(s_2))| = |\phi_1(s_1) - \phi_2(s_2)|$. This completes the proof of \eqref{eq:jacobian}. 
\end{proof}

\bibliographystyle{plain}
\bibliography{references}

\end{document}